\documentclass[psamsfonts]{amsart}

\usepackage{amsmath, amsthm, amsfonts, amssymb, graphicx, url, tikz}

\makeatletter 
\newtheorem*{rep@theorem}{\rep@title}
\newcommand{\newreptheorem}[2]{%
\newenvironment{rep#1}[1]{%
 \def\rep@title{#2 \ref{##1}}%
 \begin{rep@theorem}}%
 {\end{rep@theorem}}}
\makeatother

\newtheorem{theorem}{Theorem}
\numberwithin{theorem}{section}
\newreptheorem{theorem}{Theorem}
\newtheorem{lemma}[theorem]{Lemma}

\newtheorem{corollary}[theorem]{Corollary}
\newtheorem{conjecture}{Conjecture}
\numberwithin{conjecture}{section}

\theoremstyle{definition}

\newtheorem{question}[conjecture]{Question}

\newcommand{\del}{\partial}
\newcommand{\abs}[1]{\left\lvert #1 \right\rvert}
\newcommand{\co}{\colon\thinspace}

\DeclareMathOperator{\config}{config}

\DeclareMathOperator{\cell}{cell}
\DeclareMathOperator{\key}{key}

\DeclareMathOperator{\FI}{FI}
\DeclareMathOperator{\no}{no}
\DeclareMathOperator{\desc}{desc}

\begin{document}

\title[Representation stability for disks in a strip]{Generalized representation stability for disks in a strip and no-$k$-equal spaces}
\author{Hannah Alpert}
\address{University of British Columbia, 1984 Mathematics Rd, Vancouver, BC V6T\;1Z2, Canada}
\email{hcalpert@math.ubc.ca}

\begin{abstract}
For fixed $j$ and $w$, we study the $j$th homology of the configuration space of $n$ labeled disks of width $1$ in an infinite strip of width $w$.  As $n$ grows, the homology groups grow exponentially in rank, suggesting a generalized representation stability as defined by Church--Ellenberg--Farb and Ramos.  We prove this generalized representation stability for the strip of width $2$, leaving open the case of $w > 2$.  We also prove it for the configuration space of $n$ labeled points in the line, of which no $k$ are equal.
\end{abstract}

\maketitle

\section{Introduction}

The configuration space of $n$ labeled unit-diameter disks in an infinite strip of width $w$ is denoted $\config(n, w)$; Figure~\ref{fig-def} depicts an example configuration.  Specifically, parametrizing the configurations in terms of the centers of the disks, $\config(n, w)$ is the set of points $(x_1, y_1, \ldots, x_n, y_n) \in \mathbb{R}^2$, such that $(x_i - x_j)^2 + (y_i - y_j)^2 \geq 1$ for all $i$ and $j$, and such that $\frac{1}{2} \leq y_i \leq w - \frac{1}{2}$ for all $i$.  We would like to describe the topology of $\config(n, w)$.

\begin{figure}[h!]
\begin{center}
\begin{tikzpicture}[scale=.8, emp/.style={inner sep = 0pt, outer sep = 0pt}, >=stealth]
\draw (-1, 0)--(12, 0);
\draw (-1, 3)--(12, 3);
\node at (0, 1.5) {$w$};
\draw[->] (0, 2)--(0, 3);
\draw[->] (0, 1)--(0, 0);
\node[emp] (d2) at (2, 1) {$2$};
\node[emp] (d3) at (4, .5) {$3$};
\node[emp] (d1) at (4, 1.5) {$1$};
\node[emp] (d4) at (4, 2.5) {$4$};
\node[emp] (dn) at (6, 2) {$n$};
\node at (8, 1.5) {$\cdots$};
\node[emp] (d6) at (10, .5) {$6$};
\node[emp] (d5) at (10.707, 1.207) {$5$};
\draw (d2) circle (.5);
\draw (d3) circle (.5);
\draw (d1) circle (.5);
\draw (d4) circle (.5);
\draw (dn) circle (.5);
\draw (d6) circle (.5);
\draw (d5) circle (.5);
\end{tikzpicture}
\end{center}
\caption{The configuration space $\config(n, w)$ is the set of ways to arrange $n$ disjoint labeled disks of width $1$ in $\mathbb{R} \times [0, w]$.}\label{fig-def}
\end{figure}
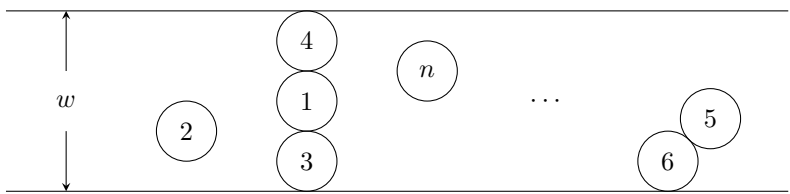

The topological study of disk configuration spaces was initiated by Baryshnikov, Bubenik, and Kahle in~\cite{Baryshnikov13}.  It is closely related to topological robotics and motion planning, described in Farber's survey~\cite{Farber08}.  Earlier, disk configuration spaces were studied probabilistically, in the context of the hard spheres gas model.  In the survey~\cite{Diaconis09}, Diaconis describes that work in statistical mechanics by physicists and materials scientists, and Carlsson et al.\ review the physics literature in \cite{Carlsson12}.

In addition to the study of disks in~\cite{Baryshnikov13}, others have studied the topology of configuration spaces of various identical rigid objects in various shapes of container, such as in \cite{Alpert17-disks}, \cite{Deeley11}, and \cite{Kusner16}.  The choice of disks in a strip is geometrically simplest among the possibilities.  The configuration space of $n$ labeled unit-diameter disks in the plane, which we denote by $\config(n)$, is homotopy equivalent to the configuration space of points in the plane, which is well understood (see, for instance, \cite{Arnold69} or~\cite{Sinha13}).  In fact, if the strip is wide compared to the number of disks, specifically if $w \geq n$, then $\config(n, w)$ and $\config(n)$ are homotopy equivalent.  For $w < n$, though, the strip shrinks the configuration space in a way that adds topology.

The paper~\cite{Alpert19} introduces the spaces $\config(n, w)$ and asks, for fixed $j$ and $w$, how does $H_j(\config(n, w))$ depend on $n$?  That paper estimates the dimension of $H_j(\config(n, w))$  up to a constant factor; it turns out to be exponential in $n$ unless the strip is wide compared to $j$.  The present paper continues the study of how $H_j(\config(n, w))$ depends on $n$, putting it into the framework of generalized representation stability as introduced by Ramos in~\cite{Ramos17}.  The goal is to give algebraic relationships between the various homology groups in a way that recovers the asymptotic results about their dimension growth.

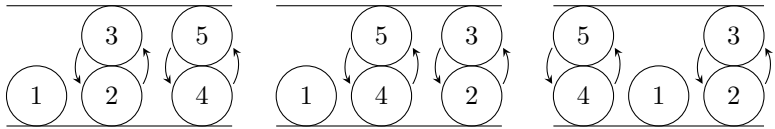
\begin{figure}
\begin{center}
\begin{tikzpicture}[scale=.8, emp/.style={inner sep = 0pt, outer sep = 0pt}, >=stealth]
\draw (-.25, 0)--(3.5, 0);
\draw (-.25, 2)--(3.5, 2);

\node (d1) at (.25, .5) {$1$};
\node (d2) at (1.5, .5) {$2$};
\node (d3) at (1.5, 1.5) {$3$};
\node (d4) at (3, .5) {$4$};
\node (d5) at (3, 1.5) {$5$};
\draw (d1) circle (.5);
\draw (d2) circle (.5);
\draw (d3) circle (.5);
\draw (d4) circle (.5);
\draw (d5) circle (.5);

\draw[->] (1.5+.52, .7) arc (-30:30:.6);
\draw[->] (1.5-.52, 1.3) arc (150:210:.6);
\draw[->] (3+.52, .7) arc (-30:30:.6);
\draw[->] (3-.52, 1.3) arc (150:210:.6);
\end{tikzpicture}\hspace{10pt}
\begin{tikzpicture}[scale=.8, emp/.style={inner sep = 0pt, outer sep = 0pt}, >=stealth]
\draw (-.25, 0)--(3.5, 0);
\draw (-.25, 2)--(3.5, 2);

\node (d1) at (.25, .5) {$1$};
\node (d2) at (1.5, .5) {$4$};
\node (d3) at (1.5, 1.5) {$5$};
\node (d4) at (3, .5) {$2$};
\node (d5) at (3, 1.5) {$3$};
\draw (d1) circle (.5);
\draw (d2) circle (.5);
\draw (d3) circle (.5);
\draw (d4) circle (.5);
\draw (d5) circle (.5);

\draw[->] (1.5+.52, .7) arc (-30:30:.6);
\draw[->] (1.5-.52, 1.3) arc (150:210:.6);
\draw[->] (3+.52, .7) arc (-30:30:.6);
\draw[->] (3-.52, 1.3) arc (150:210:.6);
\end{tikzpicture}\hspace{10pt}
\begin{tikzpicture}[scale=.8, emp/.style={inner sep = 0pt, outer sep = 0pt}, >=stealth]
\draw (-1.75, 0)--(1.75, 0);
\draw (-1.75, 2)--(1.75, 2);

\node (d1) at (0, .5) {$1$};
\node (d2) at (1.25, .5) {$2$};
\node (d3) at (1.25, 1.5) {$3$};
\node (d4) at (-1.25, .5) {$4$};
\node (d5) at (-1.25, 1.5) {$5$};
\draw (d1) circle (.5);
\draw (d2) circle (.5);
\draw (d3) circle (.5);
\draw (d4) circle (.5);
\draw (d5) circle (.5);

\draw[->] (1.25+.52, .7) arc (-30:30:.6);
\draw[->] (1.25-.52, 1.3) arc (150:210:.6);
\draw[->] (-1.25+.52, .7) arc (-30:30:.6);
\draw[->] (-1.25-.52, 1.3) arc (150:210:.6);
\end{tikzpicture}
\end{center}
\caption{A single homology class in the plane can correspond to several homology classes in the strip.  Pictured are three maps $S^1 \times S^1 \rightarrow \config(5, 2)$ that give distinct (indeed, linearly independent) classes in $H_2(\config(5, 2))$.}\label{fig-lindep}
\end{figure}

Why does $\dim H_j(\config(n, w))$ grow exponentially in $n$?  We know that when the strip is replaced by the plane, $\dim H_j(\config(n))$ grows polynomially in $n$.  Why does the subspace $\config(n, w)$ have so much more homology?  As depicted in Figure~\ref{fig-lindep}, cycles that are homologous in $H_j(\config(n))$ may not be homologous in $H_j(\config(n, w))$, because the strip is too narrow to let the various clusters of disks homotope past each other.  

The rough idea of the exponential growth in $\dim H_j(\config(n, w))$ is as follows.  In $\config(n, w)$, it is possible for $w$ disks to revolve around each other to make a $(w-1)$--cycle, forming a ``barrier'' that no other disks can pass.  (For $w = 2$, the barriers would be the circling pairs shown in Figure~\ref{fig-lindep}.)  Very broadly, the generators of $H_j(\config(n, w))$ look like sequences of barriers with smaller clusters of disks in between; if there are $b$ barriers, they divide the strip into $b+1$ intervals, so the remaining disks each have $b+1$ choices for which interval to be in.  This gives roughly $(b+1)^n$ linearly independent homology classes in $H_j(\config(n, w))$.

In some sense, once $n$ is large enough, incrementing $n$ by $1$ does not meaningfully change the structure of $H_j(\config(n, w))$---the extra disk has a choice of $b+1$ intervals to be placed in, and nothing else happens.  The framework of representation stability, first introduced in~\cite{Church13}, is well suited to situations such as this one.  In fact, one of the favorite examples of representation stability is the sequence $H_j(\config(n))$, as $n$ varies and $j$ stays fixed.  Each space $\config(n)$ has an action of $S_n$ by permuting the disks, so each $H_j(\config(n))$ is a representation of $S_n$.  Representation stability, very broadly, says that for sufficiently large $n$, incrementing $n$ by $1$ changes the $S_n$--representation $H_j(\config(n))$ in the most trivial way to give the $S_{n+1}$--representation $H_j(\config(n+1))$.  The topological reason for this is that $H_j(\config(n))$ turns out to be generated by cycles in which at most $2j$ of the disks move at all.  So, for $n > 2j$, the extra disks do nothing but sit on the side.

The formal way to talk about extra disks sitting on the side is to say that $H_j(\config(n))$ is a finitely generated $\FI$--module, first defined in~\cite{Church15} by Church, Ellenberg, and Farb.  The category $\FI$ is defined to have one object $[n] = \{1, 2, \ldots, n\}$ for each natural number $n$, and the morphisms between these objects are the injections.  For instance, the set of $\FI$--morphisms from $[n]$ to $[n]$ is the symmetric group $S_n$.  An $\FI$--module $M$ over a commutative ring $k$ is a functor from $\FI$ to $k$--modules; that is, we have a $k$--module $M_n$ for each $n$, and for each injection $[n] \rightarrow [m]$ we have a corresponding homomorphism $M_n \rightarrow M_m$.  In this paper we only consider the case $k = \mathbb{Z}$, where each of the modules is an abelian group.  For any $j$, the homology groups $M_n = H_j(\config(n))$ form an $\FI$--module over $\mathbb{Z}$; given an injection $\varphi\co [n] \rightarrow [m]$ we have a map $\varphi_*\co H_j(\config(n)) \rightarrow H_j(\config(m))$ given by the map of spaces that relabels the disks $1, 2, \ldots, n$ by $\varphi(1), \varphi(2), \ldots, \varphi(n)$ and places $m-n$ disks with the remaining labels off to the side, as shown in Figure~\ref{fig-fi-plane}.  An $\FI$--module is \textit{\textbf{finitely generated}} if there exists a finite set of elements $x_1, \ldots, x_r \in \bigsqcup_{n = 1}^\infty M_n$ such that the only $\FI$--submodule of $M$ containing $x_1, \ldots, x_r$ is $M$ itself.  Our $\FI$-module $H_j(\config(n))$ is finitely generated by classes in $H_j(\config(2j))$.

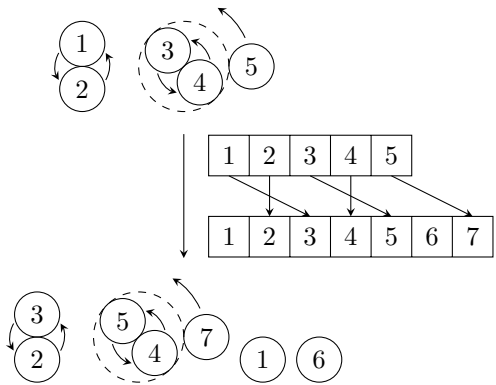
\begin{figure}
\begin{center}
\begin{tikzpicture}[scale=.6, emp/.style={inner sep = 0pt, outer sep = 0pt}, >=stealth]
\node (d3) at (-2.25, 6+.5) {$1$};
\node (d2) at (-2.25, 6-.5) {$2$};
\node (d4) at (.354, 6-.354) {$4$};
\node (d5) at (-.354, 6+.354) {$3$};
\node (d7) at (1.5, 6+0) {$5$};
\draw (d2) circle (.5);
\draw (d3) circle (.5);
\draw (d4) circle (.5);
\draw (d5) circle (.5);
\draw (d7) circle (.5);

\draw[->] (-2.25+.52, 6-1+.7) arc (-30:30:.6);
\draw[->] (-2.25-.52, 6-1+1.3) arc (150:210:.6);

\draw[->] (.58, 6+.155) arc (15:75:.6);
\draw[->] (-.58,6 -.155) arc (-165:-105:.6);
\draw[dashed] (0, 6+0) circle (1);
\draw[->] (1.359, 6+.634) arc (25:60:1.5);


\node (d3) at (-1-2.25, .5) {$3$};
\node (d2) at (-1-2.25, -.5) {$2$};
\node (d4) at (-1+.354, -.354) {$4$};
\node (d5) at (-1-.354, .354) {$5$};
\node (d7) at (-1+1.5, 0) {$7$};
\node (d1) at (-1+2.75, -.5) {$1$};
\node (d6) at (-1+4, -.5) {$6$};
\draw (d1) circle (.5);
\draw (d2) circle (.5);
\draw (d3) circle (.5);
\draw (d4) circle (.5);
\draw (d5) circle (.5);
\draw (d6) circle (.5);
\draw (d7) circle (.5);

\draw[->] (-1-2.25+.52, -1+.7) arc (-30:30:.6);
\draw[->] (-1-2.25-.52, -1+1.3) arc (150:210:.6);

\draw[->] (-1+.58, .155) arc (15:75:.6);
\draw[->] (-1-.58, -.155) arc (-165:-105:.6);
\draw[dashed] (-1+0, 0) circle (1);
\draw[->] (-1+1.359, .634) arc (25:60:1.5);


\draw[->] (0, 4.5)--(0, 1.75);
\node[emp] at (1, 4.025) {$1$};
\node[emp] at (1+.9, 4.025) {$2$};
\node[emp] at (1+2*.9, 4.025) {$3$};
\node[emp] at (1+3*.9, 4.025) {$4$};
\node[emp] at (1+4*.9, 4.025) {$5$};
\draw (1-.45, 4.025-.45)--(1+4*.9+.45, 4.025-.45)--(1+ 4*.9+.45, 4.025+.45)--(1-.45, 4.025+.45)--cycle
(1+.45, 4.025-.45)--(1+.45, 4.025+.45)
(1+.9+.45, 4.025-.45)--(1+.9+.45, 4.025+.45)
(1+2*.9+.45, 4.025-.45)--(1+2*.9+.45, 4.025+.45)
(1+3*.9+.45, 4.025-.45)--(1+3*.9+.45, 4.025+.45);

\node[emp] at (1, 2.225) {$1$};
\node[emp] at (1+.9, 2.225) {$2$};
\node[emp] at (1+2*.9, 2.225) {$3$};
\node[emp] at (1+3*.9, 2.225) {$4$};
\node[emp] at (1+4*.9, 2.225) {$5$};
\node[emp] at (1+5*.9, 2.225) {$6$};
\node[emp] at (1+6*.9, 2.225) {$7$};
\draw (1-.45, 2.225-.45)--(1+6*.9+.45, 2.225-.45)--(1+6*.9+.45, 2.225+.45)--(1-.45, 2.225+.45)--cycle
(1+.45, 2.225-.45)--(1+.45, 2.225+.45)
(1+.9+.45, 2.225-.45)--(1+.9+.45, 2.225+.45)
(1+2*.9+.45, 2.225-.45)--(1+2*.9+.45, 2.225+.45)
(1+3*.9+.45, 2.225-.45)--(1+3*.9+.45, 2.225+.45)
(1+4*.9+.45, 2.225-.45)--(1+4*.9+.45, 2.225+.45)
(1+5*.9+.45, 2.225-.45)--(1+5*.9+.45, 2.225+.45);

\draw[->] (1, 4.025-.45)--(1+2*.9, 2.225+.45);
\draw[->] (1+.9, 4.025-.45)--(1+.9, 2.225+.45);
\draw[->] (1+2*.9, 4.025-.45)--(1+4*.9, 2.225+.45);
\draw[->] (1+3*.9, 4.025-.45)--(1+3*.9, 2.225+.45);
\draw[->] (1+4*.9, 4.025-.45)--(1+6*.9, 2.225+.45);

\end{tikzpicture}
\end{center}
\caption{To have an $\FI$--module structure on $H_3(\config(n))$, there must be a map from $H_3(\config(5))$ to $H_3(\config(7))$ for each injection from $[5]$ to $[7]$.  Pictured are one class in $H_3(\config(5))$, one injection, and the class in $H_3(\config(7))$ that results.}\label{fig-fi-plane}
\end{figure}

The fact that $H_j(\config(n))$ is a finitely generated $\FI$--module implies that its dimension grows polynomially in $n$.  Roughly, to find generators for $H_j(\config(n))$, we take $\binom{n}{2j}$ copies of each generator of $H_j(\config(2j))$, one for each choice of which disks do and do not move.  In contrast, for disks in a strip, the homology groups $H_j(\config(n, w))$ have dimensions that grow exponentially in $n$ and thus cannot be finitely generated $\FI$--modules.  The reason is the same as the reason for exponential growth: when we add a disk there is a choice of which barriers to insert it between.

The appropriate algebraic notion for $H_j(\config(n, w))$ is that of a finitely generated $\FI_d$--module.  The best example for understanding the idea of an $\FI_d$--module is the $j$th homology of the configuration space of $n$ disks on the disjoint union of $d$ planes.  Each additional disk can be added to any of the $d$ planes.  

In~\cite{Ramos17}, Ramos introduces $\FI_d$--modules and shows that finitely generated $\FI_d$--modules satisfy a notion of generalized representation stability, and in~\cite{Ramos19} he shows that the homology groups of a certain kind of graph configuration space are finitely generated $\FI_d$--modules.  The category $\FI_d$, like $\FI$, has one object $[n]$ for each natural number $n$.  The morphisms are pairs $(\varphi, c)$, where $\varphi$ is an injection, say, from $[n]$ to $[m]$, and $c$ is a $d$--coloring on the complement of the image of $\varphi$; that is, $c$ is a map from $[m] \setminus \varphi([n])$ to a set with $d$ elements such as $\{0, 1, \ldots, d-1\}$.  An $\FI_d$--module is a functor from $\FI_d$ to modules.  Figure~\ref{fig-strip-morph} sketches the $\FI_{j+1}$--module structure for $H_j(\config(n, 2))$: the colors of the disks, shown in the picture as the numbers in the diamonds, indicate where to insert the disks between the barriers.

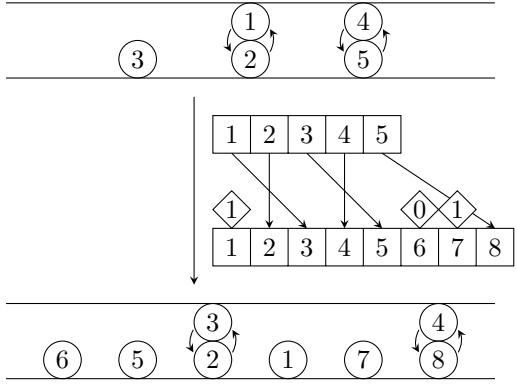
\begin{figure}
\begin{center}
\begin{tikzpicture}[scale=.5, emp/.style={inner sep = 0pt, outer sep = 0pt}, >=stealth]
\node[emp] at (1, 1.5) {$1$};
\node[emp] at (2, 1.5) {$2$};
\node[emp] at (3, 1.5) {$3$};
\node[emp] at (4, 1.5) {$4$};
\node[emp] at (5, 1.5) {$5$};
\draw (.5, 1)--(5.5, 1)--(5.5, 2)--(.5, 2)--cycle (1.5, 1)--(1.5, 2) (2.5, 1)--(2.5, 2) (3.5, 1)--(3.5, 2) (4.5, 1)--(4.5, 2);

\node[emp] at (1, -1.5) {$1$};
\node[emp] at (2, -1.5) {$2$};
\node[emp] at (3, -1.5) {$3$};
\node[emp] at (4, -1.5) {$4$};
\node[emp] at (5, -1.5) {$5$};
\node[emp] at (6, -1.5) {$6$};
\node[emp] at (7, -1.5) {$7$};
\node[emp] at (8, -1.5) {$8$};
\draw (.5, -2)--(8.5, -2)--(8.5, -1)--(.5, -1)--cycle (1.5, -2)--(1.5, -1) (2.5, -2)--(2.5, -1) (3.5, -2)--(3.5, -1) (4.5, -2)--(4.5, -1) (5.5, -2)--(5.5, -1) (6.5, -2)--(6.5, -1) (7.5, -2)--(7.5, -1);

\draw[->] (1, 1)--(3, -1);
\draw[->] (2, 1)--(2, -1);
\draw[->] (3, 1)--(5, -1);
\draw[->] (4, 1)--(4, -1);
\draw[->] (5, 1)--(8, -1);

\node[emp] at (1, -.5) {$1$};
\draw (1, -1)--(1.5, -.5)--(1, 0)--(.5, -.5)--cycle;
\draw[fill=white] (6, -1)--(6.5, -.5)--(6, 0)--(5.5, -.5)--cycle;
\node[emp] at (6, -.5) {$0$};
\draw[fill=white] (7, -1)--(7.5, -.5)--(7, 0)--(6.5, -.5)--cycle;
\node[emp] at (7, -.5) {$1$};

\draw[->] (0, 2.5)--(0, -2.5);


\draw (-5, -3)--(8, -3);
\draw (-5, -5)--(8, -5);
\node[emp] (d6) at (-3.5, -4.5) {$6$};
\node[emp] (d5) at (-1.5, -4.5) {$5$};
\node[emp] (d2) at (.5, -4.5) {$2$};
\node[emp] (d3) at (.5, -3.5) {$3$};
\node[emp] (d1) at (2.5, -4.5) {$1$};
\node[emp] (d7) at (4.5, -4.5) {$7$};
\node[emp] (d8) at (6.5, -4.5) {$8$};
\node[emp] (d4) at (6.5, -3.5) {$4$};

\draw (d1) circle (.5);
\draw (d2) circle (.5);
\draw (d3) circle (.5);
\draw (d4) circle (.5);
\draw (d5) circle (.5);
\draw (d6) circle (.5);
\draw (d7) circle (.5);
\draw (d8) circle (.5);

\draw[->] (.5+.52, -5+.7) arc (-30:30:.6);
\draw[->] (.5-.52, -5+1.3) arc (150:210:.6);
\draw[->] (6.5+.52, -5+.7) arc (-30:30:.6);
\draw[->] (6.5-.52, -5+1.3) arc (150:210:.6);


\draw (-5, 3)--(8, 3);
\draw (-5, 5)--(8, 5);
\node[emp] (d5) at (-1.5, 3.5) {$3$};
\node[emp] (d2) at (1.5, 3.5) {$2$};
\node[emp] (d3) at (1.5, 4.5) {$1$};
\node[emp] (d8) at (4.5, 3.5) {$5$};
\node[emp] (d4) at (4.5, 4.5) {$4$};

\draw (d2) circle (.5);
\draw (d3) circle (.5);
\draw (d4) circle (.5);
\draw (d5) circle (.5);
\draw (d8) circle (.5);

\draw[->] (1.5+.52, 3+.7) arc (-30:30:.6);
\draw[->] (1.5-.52, 3+1.3) arc (150:210:.6);
\draw[->] (4.5+.52, 3+.7) arc (-30:30:.6);
\draw[->] (4.5-.52, 3+1.3) arc (150:210:.6);

\end{tikzpicture}
\end{center}
\caption{When applying this $\FI_3$--morphism to a class in $H_2(\config(5, 2))$, we insert the disks with color--$k$ labels immediately after the $k$th circling pair.}\label{fig-strip-morph}
\end{figure}

We conjecture that for each $j$ and $w$, the sequence $H_j(\config(n, w))$ forms a finitely generated $\FI_d$--module for $d = 1 + \left\lfloor \frac{j}{w-1}\right\rfloor$, that is, $d$ is one more than the maximum possible number of barriers.  In this paper we prove the statement for $w = 2$; we explore in Section~\ref{sec-conclusion} which aspects of the proof seem harder to adapt for $w>2$.

\begin{reptheorem}{thm-fid}
For any $j$, the homology groups $H_j(\config(n, 2))$ form a finitely generated $\FI_{j+1}$--module over $\mathbb{Z}$.
\end{reptheorem}

The same techniques allow us to prove a similar result for a family of spaces that are closely related to the configuration spaces of disks in a strip, but are much more well-studied.  The no--$k$--equal space of the line, also known as the complement of the $k$--equal subspace arrangement, was introduced by Bj\"orner and Welker in~\cite{Bjorner95} and is the set of $n$--tuples of points in $\mathbb{R}$ such that no $k$ of them are equal.  We denote this space by $\no_k(n, \mathbb{R})$ and think of it as a configuration space of points in the line.  There is a map $\config(n, w) \rightarrow \no_{w+1}(n, \mathbb{R})$ sending each configuration to the $n$--tuple of $x$--coordinates of the centers of the disks, and the induced map on homology $H_j(\config(n, w)) \rightarrow H_j(\no_{w+1}(n, \mathbb{R}))$ is projection to a direct summand, as we show in Corollary~\ref{cor}.

The homology of $\no_k(n, \mathbb{R})$ grows exponentially in $n$ for the same reason that the homology of $\config(n, w)$ does: a cluster of $k$ points among the $n$ points can form a $(k-2)$--cycle that acts as a barrier, and the remaining points cannot cross from the left of the barrier to the right of the barrier.  Unlike in the case of $\config(n, w)$, for $\no_k(n, \mathbb{R})$ we can prove for all $k$ that the homology groups give $\FI_d$--modules.  Our results recover the computation of homology of $\no_k(n, \mathbb{R})$ from~\cite{Bjorner95}.

\begin{reptheorem}{thm-nok-fid}
For any $j \geq 0$ and $k \geq 2$, the homology groups $H_j(\no_k(n, \mathbb{R}))$ are zero unless $j$ is a multiple of $k-2$.  If $j = b(k-2)$ for some integer $b$, then the homology groups $H_j(no_k(n, \mathbb{R}))$ form a finitely generated $\FI_{b+1}$--module over $\mathbb{Z}$.
\end{reptheorem}

In Section~\ref{sec-complex} we give cell complexes $\cell(n, w)$ and $\desc(n, k-1)$ that are homotopy equivalent to $\config(n, w)$ and $\no_k(n, \mathbb{R})$, respectively.  In Section~\ref{sec-gradient} we apply discrete Morse theory to the cell complexes: we construct discrete gradient vector fields that allow us to collapse the cell complexes and eliminate most of the cells.  In Section~\ref{sec-basis} we construct a $\mathbb{Z}$--basis for each homology group, indexed by the critical cells of our discrete vector field.  In Section~\ref{sec-abstract} we prove a general lemma about how to specify an $\FI_d$--module.  In Section~\ref{sec-finish} we show that our homology groups satisfy the hypothesis of this lemma and thus form $\FI_d$--modules, and we verify that these $\FI_d$--modules are finitely generated.  In Section~\ref{sec-conclusion} we conclude by speculating about the conjectured generalization for strips of width $w > 2$.

\emph{Acknowledgments.}  This work was supported by the National Science Foundation under Award No.~DMS-1802914.  I am very grateful to Andy Putman, Nate Harman, Jenny Wilson, John Wiltshire-Gordon, and Eric Ramos, who all pointed me toward relevant and accessible information about representation stability; as someone completely unfamiliar with it, I would not have known where to start otherwise.  I also had many useful conversations with Matt Kahle about this material.

\section{Cells labeled by symbols of blocks}\label{sec-complex}

This paper is based on the technique of the paper~\cite{Alpert19}, which is to replace the configuration space $\config(n, w)$ by a homotopy-equivalent cell complex $\cell(n, w)$ and to estimate the homology by doing combinatorics (specifically, discrete Morse theory) on the cell complex.  We use the same cell complex $\cell(n, w)$ as in that paper, and we use the same method to find a cell complex $\desc(n, k-1)$ that is homotopy equivalent to the no--$k$--equal space $\no_k(n, \mathbb{R})$.  In the remainder of this section we define the complexes $\cell(n, w)$ and $\desc(n, w)$, and we prove that $\desc(n, k-1)$ is homotopy equivalent to $\no_k(n, \mathbb{R})$ by adapting the method of~\cite{Alpert19}.

The cell complex $\cell(n, w)$ is defined as a subcomplex of a cell complex $\cell(n)$ described by~\cite{Blagojevic14}.  In $\cell(n)$, every cell is labeled by a \textit{\textbf{symbol}}, which consists of a string of numbers and vertical bars, such that the numbers form a permutation of the numbers $1$ through $n$, and each vertical bar is both immediately preceded and immediately followed by a number.  Thinking of the numbers as the labels of the disks in $\config(n)$, we sometimes refer to the numbers in a symbol as labels.  Each substring between one vertical bar and the next (or before the first bar or after the last bar) is called a \textit{\textbf{block}}.  We think of the elements of each block as the labels of disks in a vertical stack in $\config(n, w)$, as in Figure~\ref{fig-symbol}.

\begin{figure}
\begin{center}
\begin{tikzpicture}[scale=.5, emp/.style={inner sep = 0pt, outer sep = 0pt}, >=stealth]
\node (a1) at (0, .5) {$7$};
\node (a2) at (0, 1.5) {$8$};
\node (a3) at (1.5, .5) {$6$};
\node (a4) at (3.75, 1.5) {$1$};
\node (a5) at (3.75, 3.5) {$5$};
\node (a6) at (3.75, 4.5) {$4$};
\node (a7) at (3.75, .5) {$3$};
\node (a8) at (3.75, 2.5) {$9$};
\draw (a1) circle (.5);
\draw (a2) circle (.5);
\draw (a3) circle (.5);
\draw (a4) circle (.5);
\draw (a5) circle (.5);
\draw (a6) circle (.5);
\draw (a7) circle (.5);
\draw (a8) circle (.5);
\draw (-1, 0)--(5.5, 0) (-1, 5)--(5.5, 5);
\end{tikzpicture}\hspace{20pt}
\begin{tikzpicture}[scale=.5, emp/.style={inner sep = 0pt, outer sep = 0pt}, >=stealth]
\node (a1) at (0, .5) {$7$};
\node (a2) at (0, 1.5) {$8$};
\node (a3) at (1.5, .5) {$6$};
\node (a4) at (3, .5) {$1$};
\node (a5) at (3, 1.5) {$5$};
\node (a6) at (3, 2.5) {$4$};
\node (a7) at (4.5, .5) {$3$};
\node (a8) at (4.5, 1.5) {$9$};
\draw (a1) circle (.5);
\draw (a2) circle (.5);
\draw (a3) circle (.5);
\draw (a4) circle (.5);
\draw (a5) circle (.5);
\draw (a6) circle (.5);
\draw (a7) circle (.5);
\draw (a8) circle (.5);
\draw (-1, 0)--(6, 0) (-1, 5)--(6, 5);
\end{tikzpicture}
\end{center}
\caption{We can imagine each symbol of $\cell(n, w)$ as a configuration in $\config(n, w)$ where the numbers in each block are the labels in a column of disks.  Pictured are configurations representing the symbol $8\ 7\ \vert\ 6\ \vert\ 4\ 5\ 9\ 1\ 3$ and its face $8\ 7\ \vert\ 6\ \vert\ 4\ 5\ 1\ \vert\ 9\ 3$.}\label{fig-symbol}
\end{figure}
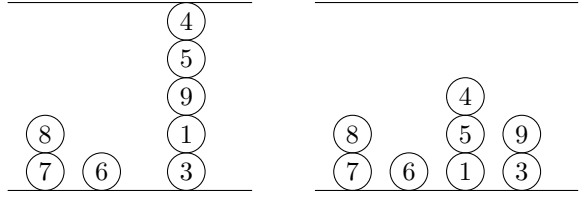

As shown in~\cite{Blagojevic14}, there is a way to form $\cell(n)$ as a polyhedral cell complex in which the cells are labeled by these symbols, with the following incidence relation.  Given two cells $f$ and $g$, we have that $f$ is a top-dimensional face of $g$ if and only if the symbol of $g$ can be obtained from the symbol of $f$ by removing a bar and combining the adjacent two blocks by a shuffle that preserves the ordering of the numbers in each of the two blocks.  The dimension of a cell is equal to $n-1$ minus the number of bars; equivalently, the dimension can be obtained by adding up $1$ less than the block size, for each block.  

To say that the cell complex is polyhedral means that it can be realized as a set of convex polytopes in Euclidean space, such that for each polytope, each of its faces (of any dimension) is also in the set, and the intersection of any two polytopes (if nonempty) is a shared face.  In this paper we do not work with the embedding of $\cell(n)$ in Euclidean space, only with the symbols of the cells, so in order to compute the homology with $\mathbb{Z}$--coefficients, we need to specify orientations and signs.

To specify the signs of the incidences in $\cell(n)$, we use the structure of the cells as products of permutahedra.  By an \textit{\textbf{injected cell}} we mean the result of taking any symbol in $\cell(n)$ and any injection from $[n]$ to a larger set $[m]$, and applying the injection to every number that appears in the symbol.  Given two injected cells with disjoint sets of labels, we can take the \textit{\textbf{concatenation product}} by writing the two symbols next to each other with a vertical bar in between.  In this way, every cell in $\cell(n)$ is a concatenation product of smaller-dimensional injected cells, except for the top-dimensional cells in $\cell(n)$, which have no vertical bars.

The signs are defined as follows.  If $g$ is a single block---that is, an injected cell with no vertical bars---and $f$ is a top-dimensional face of $g$, then we define the coefficient of $f$ in $\del g$ to be the sign of the permutation that results from deleting the bar in $f$ and not reshuffling.  To define signs on concatenation products, we use the following Leibniz rule: if $g_1$ and $g_2$ are injected cells with disjoint sets of labels, then
\[\del(g_1\ \vert\ g_2) = \del g_1\ \vert\ g_2 + (-1)^{b(g_1)}g_1 \vert\ \del g_2,\]
where $b(g_1)$ denotes the number of blocks in $g_1$.

We can check that these signs are consistent by verifying that $\del^2 = 0$.

\begin{lemma}
The differential on $\cell(n)$ satisfies $\del^2 = 0$.
\end{lemma}

\begin{proof}
Let $g$ be a cell in $\cell(n)$.  First we suppose that $g$ is a single block.  Let $e$ be a codimension--$2$ face of $g$.  Then $e = e_1\ \vert\ e_2\ \vert\ e_3$, where $e_1$, $e_2$, and $e_3$ are blocks.  Then there are two intermediate faces between $e$ and $g$.  We denote them by $x = x_1\ \vert\ e_3$ and $y = e_1\ \vert\ y_2$.  We compare the sign of $e$ in $x$ times the sign of $x$ in $g$, and the sign of $e$ in $y$ times the sign of $y$ in $g$, and we show that these two products are opposite signs.  If we consider just the contribution from the signs of the permutations, both products give the sign of the permutation relating $e$ and $g$, so those contributions are equal.

For the contribution from the Leibniz rule, only the incidence between $e$ and $y$ involves splitting a block that is not the first, so that incidence has a sign contribution of $-1$ from the Leibniz rule, and all the other incidences have a sign contribution of $1$ from the Leibniz rule.  Thus, multiplying out Leibniz rule contributions with the permutation sign contributions, we see that the total coefficient of $e$ in $g$ is zero, proving that $\del^2 g = 0$ when $g$ is a single block.

If $g$ is a concatenation product $g_1 \ \vert\ g_2$, then we use induction on the number of blocks.  The Leibniz rule gives
\[\del^2(g_1 \ \vert\ g_2) = \del^2 g_1\ \vert\ g_2 + \left [(-1)^{b(g_1)} + (-1)^{b(\del g_1)}\right ]\cdot \del g_1\ \vert\ \del g_2 + g_1\ \vert\ \del^2 g_2,\]
which is indeed zero.
\end{proof}

Having described the structure of the complex $\cell(n)$, we define $\cell(n, w)$ to be the subcomplex of $\cell(n)$ consisting of all cells for which every block has at most $w$ elements.  We define $\desc(n, w)$ to be the subcomplex of $\cell(n, w)$ in which, in addition, the elements of each block appear in descending order.  The results in this paper concern $\cell(n, w)$ only for the special case of $w = 2$, but they address $\desc(n, w)$ for all $w$.

Theorem~3.1 of~\cite{Alpert19} shows that $\cell(n, w)$ is homotopy equivalent to the configuration space $\config(n, w)$ of $n$ disks in a strip of width $w$.  The strategy is to find an open cover of $\config(n, w)$ indexed by the symbols in $\cell(n, w)$, where the intersections between open sets correspond to incidences in $\cell(n, w)$; the nerve theorem then implies that $\config(n, w)$ is homotopy equivalent to the barycentric subdivision of $\cell(n, w)$, and thus is also homotopy equivalent to $\cell(n, w)$.

The no--$k$--equal space $\no_k(n, \mathbb{R})$ consists of all the elements of $\mathbb{R}^n$ such that no $k$ of the coordinates are equal.  In the remainder of this section, we mimic the proof of Theorem~3.1 of~\cite{Alpert19}, in order to verify that $\no_k(n, \mathbb{R})$ is homotopy equivalent to $\desc(n, k-1)$.

\begin{theorem}\label{thm-nok-cell}
The no--$k$--equal space $\no_k(n, \mathbb{R})$ is homotopy equivalent to the cell complex $\desc(n, k-1)$.
\end{theorem}

Given a symbol $\alpha$ of $\desc(n, k-1)$, we let $U_\alpha$ be the subset of $\mathbb{R}^n$ consisting of all points $(x_1, \ldots, x_n)$ with the following properties:
\begin{itemize}
\item Whenever two numbers $k$ and $\ell$ are in different blocks of $\alpha$ with $k$ appearing before $\ell$, we have $x_k < x_\ell$.
\item Whenever two numbers $k$ and $\ell$ are in the same block, and $k'$ and $\ell'$ are in different blocks, we have
\[\abs{x_k - x_\ell} < \abs{x_{k'} - x_{\ell'}}.\]
\end{itemize}
The sets $U_\alpha$ are open and convex in $\mathbb{R}^n$, and their union as $\alpha$ ranges over all the symbols in $\desc(n, k-1)$ is equal to $\no_k(n, \mathbb{R})$.

The nerve $N$ of the open cover $U_\alpha$ is the simplicial complex built by taking one vertex for each $\alpha$ and a simplex for each collection of open sets $U_\alpha$ that have a nonempty intersection.  Because the sets $U_\alpha$ are convex, any intersection of them is either empty or contractible.  Thus, the nerve theorem says that $\no_k(n, \mathbb{R})$ is homotopy equivalent to the nerve $N$.  The next lemma implies that $N$ is equal to the barycentric subdivision of $\desc(n, k-1)$.

\begin{lemma}\label{lem-poset}
An intersection
\[U_{\alpha_1} \cap U_{\alpha_2} \cap \cdots \cap U_{\alpha_r}\]
is nonempty if and only if the cells corresponding to $\alpha_1, \alpha_2, \ldots, \alpha_r$ form a chain under the incidence relation in $\desc(n, k-1)$.
\end{lemma}

\begin{proof}
Let $p = (p_1, \dots, p_n)$ be an element of $U_{\alpha_1} \cap U_{\alpha_2} \cap \cdots \cap U_{\alpha_r}$.  We can find the set of all $U_\alpha$ containing $p$ in the following way.  Given any real number $\rho$, we can draw the closed interval of length $\rho$ centered at each $p_1, \ldots, p_n \in \mathbb{R}$, and take the union of these intervals in $\mathbb{R}$.  Then we can cluster the indices $1, \ldots, n$ according to which connected component of the union the points $p_1, \ldots, p_n$ fall into.  Reading off these clusters from left to right, and ordering the indices within each cluster in descending order, we obtain a symbol $\alpha(\rho)$ associated to $p$ and $\rho$, as in Figure~\ref{fig-grow}.  Then the symbols $\alpha(\rho)$ for various $\rho$ form a chain under incidence in $\desc(n, k-1)$, and $p \in U_\alpha$ if and only if $\alpha = \alpha(\rho)$ for some $\rho$.  Thus, $\alpha_1, \alpha_2, \ldots, \alpha_r$ must all be part of this chain, and so they must also form a chain.

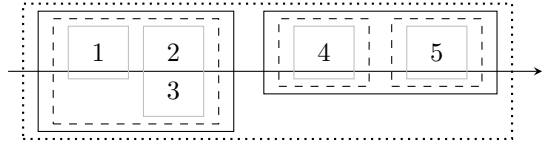
\begin{figure}
\begin{center}
\begin{tikzpicture}[scale=1, emp/.style={inner sep = 0pt, outer sep = 0pt}, >=stealth]
\draw[->] (-.2, 0)--(6.9, 0);
\node[emp] at (1, 0) [label=above:{$1$}] {};
\node[emp] at (2, 0) [label=above:{$2$}] [label=below:{$3$}] {};
\node[emp] at (4, 0) [label=above:{$4$}] {};
\node[emp] at (5.5, 0) [label=above:{$5$}] {};

\draw[gray!50] (.6, -.1)--(1.4, -.1)--(1.4, .6)--(.6, .6)--cycle (1.6, -.6)--(2.4, -.6)--(2.4, .6)--(1.6, .6)--cycle (3.6, -.1)--(4.4, -.1)--(4.4, .6)--(3.6, .6)--cycle (5.1, -.1)--(5.9, -.1)--(5.9, .6)--(5.1, .6)--cycle;
\draw[dashed] (.4, -.7)--(2.6, -.7)--(2.6, .7)--(.4, .7)--cycle (3.4, -.2)--(4.6, -.2)--(4.6, .7)--(3.4, .7)--cycle (4.9, -.2)--(6.1, -.2)--(6.1, .7)--(4.9, .7)--cycle;
\draw (.2, -.8)--(2.8, -.8)--(2.8, .8)--(.2, .8)--cycle (3.2, -.3)--(6.3, -.3)--(6.3, .8)--(3.2, .8)--cycle;
\draw[dotted, thick] (0, -.9)--(6.5, -.9)--(6.5, .9)--(0, .9)--cycle;

\end{tikzpicture}
\end{center}
\caption{Given a configuration $p$ in $\no_k(n, \mathbb{R})$, the set of symbols $\alpha$ such that $p \in U_\alpha$ forms a totally ordered chain in $\desc(n, k-1)$.  The configuration pictured is in $U_\alpha$ for $\alpha = 1 \ \vert\ 3\ 2\ \vert\ 4\ \vert\ 5$, $3\ 2\ 1\ \vert\ 4\ \vert\ 5$, $3\ 2\ 1\ \vert\ 5\ 4$, and $5\ 4\ 3\ 2\ 1$.}\label{fig-grow}
\end{figure}

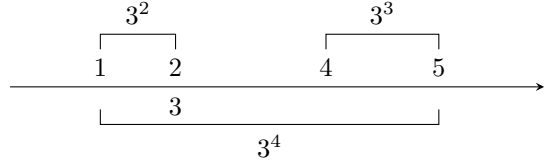
\begin{figure}
\begin{center}
\begin{tikzpicture}[scale=1, emp/.style={inner sep = 0pt, outer sep = 0pt}, >=stealth]
\draw[->] (-.2, 0)--(6.9, 0);
\node[emp] at (1, 0) [label=above:{$1$}] {};
\node[emp] at (2, 0) [label=above:{$2$}] [label=below:{$3$}] {};
\node[emp] at (4, 0) [label=above:{$4$}] {};
\node[emp] at (5.5, 0) [label=above:{$5$}] {};

\draw (1, .5)--(1, .7)--(2, .7)--(2, .5) (4, .5)--(4, .7)--(5.5, .7)--(5.5, .5) (1, -.3)--(1, -.5)--(5.5, -.5)--(5.5, -.3);
\node[emp] at (1.5, .7) [label=above:{$3^2$}] {};
\node[emp] at (4.75, .7) [label=above:{$3^3$}] {};
\node[emp] at (3.25, -.5) [label=below:{$3^4$}] {};
\end{tikzpicture}
\end{center}
\caption{Given a chain of symbols in $\desc(n, k-1)$, such as $1\ \vert\ 3\ 2\ \vert\ 4\ \vert\ 5 \prec 3\ 2\ 1\ \vert\ 4\ \vert\ 5 \prec 3\ 2\ 1\ \vert\ 5\ 4 \prec 5\ 4\ 3\ 2\ 1$, we space the numbers $1$ through $n$ in $\mathbb{R}$ such that the intervals in $\mathbb{R}$ formed by the various blocks are consecutive powers of $3$.  The resulting element of $\no_k(n, \mathbb{R})$ is in $U_\alpha$ for each $\alpha$ in the chain.}\label{fig-pow3}
\end{figure}

For the converse, suppose that $\alpha_1, \alpha_2, \ldots, \alpha_r$ form a chain in $\desc(n, k-1)$.  We need to produce a point $p$ in $U_{\alpha_1} \cap \cdots \cap U_{\alpha_r}$.  Without loss of generality, we assume that the chain is maximal in $\desc(n) = \desc(n, n)$ and that $\alpha_1, \ldots, \alpha_r$ are in order, so $\alpha_1$ has only blocks of size $1$, and getting to each symbol $\alpha_i$ from the previous symbol $\alpha_{i-1}$ corresponds to merging two consecutive blocks.  We start with $\alpha_1$ and add restrictions on the coordinates $(p_1, \ldots, p_n)$ one step at a time, so that on the $i$th step we will have fixed the differences between coordinates within each block of $\alpha_i$, but we think of the separate blocks sliding freely from side to side.  After all the steps, we will have specified the configuration $(p_1, \ldots, p_n)$ up to translating every coordinate by the same real number.

More precisely, at step $1$ we require that if $k$ appears before $\ell$ in $\alpha_1$, then $p_{k} < p_{_\ell}$, with no other restrictions.  Any such configuration is in $U_{\alpha_1}$.  At step $2$, two consecutive elements in $\alpha_1$ together form a block of size $2$, and we introduce the restriction that their coordinates have difference $9=3^2$.  Then, continuing in the same way, at step $i$ two consecutive blocks $c_{k_i}$ and $c_{k_{i}+1}$ in $\alpha_{i-1}$ merge to give $\alpha_i$.  We introduce the restriction that the difference in coordinates between the first element of $c_{k_i}$ and the last element of $c_{k_{i}+1}$---where ``first'' and ``last'' are still taken in terms of the first symbol $\alpha_1$---should be $3^i$.  Figure~\ref{fig-pow3} depicts this process of selecting the widths of the blocks.

Any configuration that satisfies the restrictions up through step $i$ and also leaves horizontal gaps larger than $3^i$ between the blocks of $\alpha_i$ is in $U_{\alpha_i}$.  Note that step $i$ sets the gap between blocks $c_{k_i}$ and $c_{k_{i}+1}$ of $\alpha_{i-1}$ to be more than $3^{i-1}$, which is what we need in order for the final configuration to be in $U_{\alpha_{i-1}}$.  This is because, if we use the word ``width'' here to mean the range of coordinates, the widths of $c_{k_i}$ and $c_{k_{i}+1}$ have been set to be distinct powers of $3$ less than $3^i$, or to be $0$ if the block has only one element.  Thus, the gap has size at most $3^i - (3^{i-1} + 3^{i-2}) > 3^{i-1}$.

In the final step, step $n$ means merging two blocks to get $\alpha_{n}$ which has only one block, and at step $n$ we set the difference between $p_s$ and $p_t$ to be $3^n$, where $s$ is the first (leftmost) number in $\alpha_1$ and $t$ is the last (rightmost) number in $\alpha_1$.  At this stage we have specified the point $p$ up to translation in $\mathbb{R}$, and it is in $U_{\alpha_1} \cap \dots \cap U_{\alpha_n}$.
\end{proof}

The lemma above gives the bulk of the proof that $\no_k(n, \mathbb{R})$ is homotopy equivalent to $\desc(n, k-1)$.

\begin{proof}[Proof of Theorem~\ref{thm-nok-cell}]
The barycentric subdivision of $\desc(n, w)$ has one vertex for every cell in $\desc(n, w)$, and one simplex for every chain of incident cells in $\desc(n, w)$.  Taking $w = k-1$, the nerve $N$ has one vertex for each $U_\alpha$, and thus for each cell in $\desc(n, k-1)$.  And, we have just shown that every set of $U_\alpha$ with nonempty intersection---corresponding to a simplex in $N$---corresponds to a chain of incident cells in $\desc(n, k-1)$, and vice versa.  Thus, $N$ is equal to the barycentric subdivision of $\desc(n, k-1)$.  The nerve lemma states that $N$ is homotopy equivalent to the union of the various $U_\alpha$, which is $\no_k(n, \mathbb{R})$.
\end{proof}

\begin{corollary}\label{cor}
For any $j \geq 0$, the homology group $H_j(\no_{w+1}(n, \mathbb{R}))$ is a direct summand of $H_j(\config(n, w))$.
\end{corollary}

\begin{proof}
Let $p \co \cell(n, w) \rightarrow \desc(n, w)$ be the cellular map that sends the cell $\alpha$ to the cell in which the numbers in each block of $\alpha$ are rearranged to be in descending order.  If $i \co \desc(n, w) \rightarrow \cell(n, w)$ is the inclusion map, then $p \circ i$ is the identity on $\desc(n, w)$.  Thus, the induced maps on homology satisfy the relation that $p_* \circ i_*$ is the identity on each $H_j(\desc(n, w))$.  These maps give a way to write $H_j(\desc(n, w))$ as a direct summand of $H_j(\cell(n, w))$, and thus give a way to write $H_j(\no_{w+1}(n, \mathbb{R}))$ as a direct summand of $H_j(\config(n, w))$.
\end{proof}

\section{Discrete gradient vector field}\label{sec-gradient}

In the next two sections, we use discrete Morse theory to compute the homology of the cell complexes $\cell(n, 2)$ and $\desc(n, w)$, which we have shown in the previous section are homotopy equivalent to the configuration spaces $\config(n, 2)$ and $\no_{w+1}(n, \mathbb{R})$.  In any cell complex, the cellular homology comes from a chain complex generated by the cells; very broadly, discrete Morse theory gives a way to decompose the chain complex as a direct sum of a chain complex that has no homology (which we discard) and a chain complex generated by a smaller subset of cells, the critical cells.  This section concerns the reduction to the smaller chain complex, and the next section shows that in fact, in the smaller chain complex all differentials are zero, so the homology has a $\mathbb{Z}$--basis in bijection with the set of critical cells.

The basic definitions in discrete Morse theory are as follows.  In any polyhedral cell complex, we say that cell $f$ is a \textit{\textbf{face}} of cell $g$ if $f$ is in the boundary of $g$ and $\dim f = \dim g - 1$, and we say that $g$ is a \textit{\textbf{coface}} of $f$ if $f$ is a face of $g$.  A \textit{\textbf{discrete vector field}} on a polyhedral cell complex is a set $V$ of pairs of cells $[f, g]$ such that $f$ is a face of $g$ and each cell can be in at most one pair; an example is shown in Figure~\ref{fig-discrete-morse}.  A discrete vector field $V$ is \textit{\textbf{gradient}} if there are no closed $V$--walks.  A $V$--walk is a sequence of pairs $[f_1, g_1], \ldots, [f_r, g_r]$ with $[f_i, g_i] \in V$, such that each $f_{i+1}$ is a face of $g_i$ other than $f_i$.  The $V$--walk is closed if $f_r = f_1$.

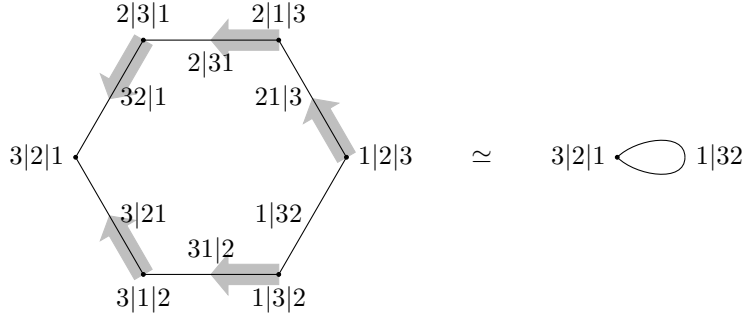
\begin{figure}
\begin{center}
\begin{tikzpicture}[scale = 1.8, emp/.style={inner sep = 0pt, outer sep = 0pt}, vert/.style={circle, draw=black, fill=black, inner sep = 0pt, minimum size = .5mm}]
\draw[draw=gray!50, fill=gray!50] (1.06495000000000, 0.0375000000000000) -- (0.877450000000000, 0.362250000000000) -- (0.942400000000000, 0.399750000000000) -- (0.750000000000000, 0.433000000000000) -- (0.682600000000000, 0.249750000000000) -- (0.747550000000000, 0.287250000000000) -- (0.935050000000000, -0.0375000000000000) --cycle;
\draw[draw=gray!50, fill=gray!50] (0.500000000000000, 0.941000000000000) -- (0.125000000000000, 0.941000000000000) -- (0.125000000000000, 1.01600000000000) -- (0.000000000000000, 0.866000000000000) -- (0.125000000000000, 0.716000000000000) -- (0.125000000000000, 0.791000000000000) -- (0.500000000000000, 0.791000000000000) --cycle;
\draw[draw=gray!50, fill=gray!50] (-0.564950000000000, 0.903500000000000) -- (-0.752450000000000, 0.578750000000000) -- (-0.817400000000000, 0.616250000000000) -- (-0.750000000000000, 0.433000000000000) -- (-0.557600000000000, 0.466250000000000) -- (-0.622550000000000, 0.503750000000000) -- (-0.435050000000000, 0.828500000000000) --cycle;
\draw[draw=gray!50, fill=gray!50] (0.500000000000000, -0.791000000000000) -- (0.125000000000000, -0.791000000000000) -- (0.125000000000000, -0.716000000000000) -- (0.000000000000000, -0.866000000000000) -- (0.125000000000000, -1.01600000000000) -- (0.125000000000000, -0.941000000000000) -- (0.500000000000000, -0.941000000000000) --cycle;
\draw[draw=gray!50, fill=gray!50] (-0.435050000000000, -0.828500000000000) -- (-0.622550000000000, -0.503750000000000) -- (-0.557600000000000, -0.466250000000000) -- (-0.750000000000000, -0.433000000000000) -- (-0.817400000000000, -0.616250000000000) -- (-0.752450000000000, -0.578750000000000) -- (-0.564950000000000, -0.903500000000000) --cycle;
\draw (1, 0)--(.5, .866)--(-.5, .866)--(-1, 0)--(-.5, -.866)--(.5, -.866)--(1, 0);
\node[vert] at (1, 0) [label=right:{$1\vert 2 \vert 3$}] {};
\node[vert] at (.5, .866) [label=above:{$2\vert 1 \vert 3$}] {};
\node[vert] at (-.5, .866) [label=above:{$2\vert 3 \vert 1$}] {};
\node[vert] at (-1, 0) [label=left:{$3\vert 2 \vert 1$}] {};
\node[vert] at (-.5, -.866) [label=below:{$3\vert 1 \vert 2$}] {};
\node[vert] at (.5, -.866) [label=below:{$1\vert 3 \vert 2$}] {};
\node[emp] at (.75, .433) [label=left: {$21\vert 3$}] {};
\node[emp] at (0, .866) [label=below: {$2\vert 31$}] {};
\node[emp] at (-.75, .433) [label=right: {$32\vert 1$}] {};
\node[emp] at (-.75, -.433) [label=right: {$3\vert 21$}] {};
\node[emp] at (0, -.866) [label=above: {$31\vert 2$}] {};
\node[emp] at (.75, -.433) [label=left: {$1\vert 32$}] {};
\node at (2, 0) {$\simeq$};
\node[vert] at (3, 0) [label=left: {$3\vert 2\vert 1$}] {};
\draw (3, 0) to [out=45, in=90] (3.5, 0);
\draw (3, 0) to [out=-45, in=-90] (3.5, 0);
\node[emp] at (3.5, 0) [label=right: {$1\vert 32$}] {};
\end{tikzpicture}
\end{center}
\caption{A \emph{discrete gradient vector field} consists of a set of disjoint pairs of cells, each pair incident and of consecutive dimensions.  The complex is homotopy equivalent to one in which the paired cells are collapsed, and only the \emph{critical} (unpaired) cells remain.}\label{fig-discrete-morse}
\end{figure}

A cell is \textit{\textbf{critical}} with respect to a discrete gradient vector field $V$ if the cell is not in any pair in $V$.  The fundamental theorem of discrete Morse theory~\cite{Forman02} states that there is a cell complex that is a strong deformation retraction of the original cell complex, in which there is one cell per critical cell of $V$.  Thus, we can compute the homology groups $H_j(\cell(n, 2))$ and $H_j(\desc(n, w)$ by defining discrete gradient vector fields and computing the homology of the collapsed chain complexes generated by the critical cells.

One way to define a discrete gradient vector field on a polyhedral cell complex is by defining a total ordering on all the cells.  Given a total ordering, the resulting vector field contains a pair $[f, g]$ if and only if both $f$ is the greatest face of $g$ and $g$ is the least coface of $f$; using the fact that the cell complex is polyhedral, one can prove that this vector field is gradient (see Lemma~3.7 of~\cite{Bauer19}).  In what follows, we define a total ordering on all of $\cell(n)$, the polyhedral complex that contains both $\cell(n, 2)$ and $\desc(n, w)$ as subcomplexes.  We use the resulting discrete gradient vector fields to compute the homology.

To describe the ordering, let $\alpha = \alpha_1 \ \vert\ \alpha_2 \ \vert\ \cdots \ \vert\ \alpha_r$ and $\beta = \beta_1\ \vert\ \beta_2\ \vert\ \cdots \ \vert\ \beta_s$ be symbols in $\cell(n)$.  We say that a block $\alpha_i$ or $\beta_i$ is a \textbf{\textit{singleton}} if it has only one element.  We say that a block $\alpha_i$ is a \textbf{\textit{follower}} if the preceding block $\alpha_{i-1}$ is a singleton less than every element of $\alpha_i$.

\begin{lemma}\label{lem-key}
There is a total ordering $\prec$ on $\cell(n)$ with the following properties.  Suppose that $\alpha$ and $\beta$ first differ at block $i$.  Then, 
\begin{enumerate}
\item If $\alpha_i$ and $\beta_i$ are both followers, if $\beta_i$ has more elements than $\alpha_i$ then $\alpha \prec \beta$.
\item If neither $\alpha_i$ nor $\beta_i$ is a follower, if $\beta_i$ has a lesser first element than $\alpha_i$ then $\alpha \prec \beta$.
\item If neither $\alpha_i$ nor $\beta_i$ is a follower, and $\alpha_i$ and $\beta_i$ have the same first element, if $\beta_i$ has more elements than $\alpha_i$ then $\alpha \prec \beta$.
\item If $\alpha_i$ is a follower and $\beta_i$ is not, then $\alpha \prec \beta$.
\end{enumerate}
\end{lemma}

\begin{proof}
To define the ordering, we first define a ``key'' function that maps each cell to an element of $\bigoplus_{i = 1}^\infty \mathbb{Z}$.  Then we order the symbols lexicographically by key, and extend this partial order arbitrarily to a total order.  For any cell $\alpha$, each block $\alpha_i$ of $\alpha$ contributes two entries to $\key(\alpha)$.  The $(2i-1)$st entry of $\key(\alpha)$ is $n+1$ minus the first element of the block $\alpha_i$ if $\alpha_i$ is not a follower, or $0$ if $\alpha_i$ is a follower.  The $(2i)$th entry of $\key(\alpha)$ is the number of elements in $\alpha_i$.  Past twice the number of blocks, all the entries of $\key(\alpha)$ are zero.

One can verify that the lexicographical ordering of keys has the properties given in the lemma statement.
\end{proof}

This total ordering gives rise to different discrete gradient vector fields on $\cell(n, 2)$ and $\desc(n, w)$.  The next two lemmas describe the set of critical cells for each.  Although each lemma only proves that every critical cell has the properties specified in the lemma, the theorems of the next section imply that the converse is also true.

\begin{lemma}\label{lem-cell-crit}
If a cell in $\cell(n, 2)$ is critical with respect to the discrete gradient vector field that comes from the total ordering from Lemma~\ref{lem-key}, then the cell has the following properties:
\begin{enumerate}
\item Every two consecutive singletons are in decreasing order.
\item If a given $2$--element block has its elements in decreasing order, then the block is a follower.
\end{enumerate}
\end{lemma}

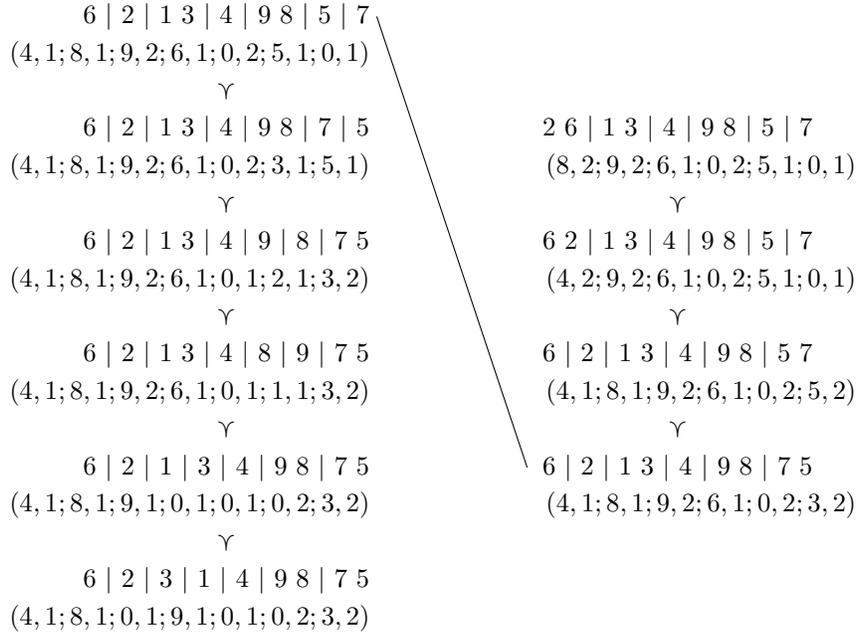
\begin{figure}
\begin{center}
\begin{tikzpicture}[scale = .5, emp/.style={inner sep = 0pt, outer sep = 0pt}]
\node[emp] at (0, 8) {$6\ \vert\ 2\ \vert\ 1\ 3\ \vert\ 4\ \vert\ 9\ 8\ \vert\ 5\ \vert\ 7$};
\node[emp] at (-1, 7) {$(4, 1; 8, 1; 9, 2; 6, 1; 0, 2; 5, 1; 0, 1)$};
\node[emp] at (0, 6) {$\curlyvee$};
\node[emp] at (0, 5) {$6\ \vert\ 2\ \vert\ 1\ 3\ \vert\ 4\ \vert\ 9\ 8\ \vert\ 7\ \vert\ 5$};
\node[emp] at (-1, 4) {$(4, 1; 8, 1; 9, 2; 6, 1; 0, 2; 3, 1; 5, 1)$};
\node[emp] at (0, 3) {$\curlyvee$};
\node[emp] at (0, 2) {$6\ \vert\ 2\ \vert\ 1\ 3\ \vert\ 4\ \vert\ 9\ \vert\ 8\ \vert\ 7\ 5$};
\node[emp] at (-1, 1) {$(4, 1; 8, 1; 9, 2; 6, 1; 0, 1; 2, 1; 3, 2)$};
\node[emp] at (0, 0) {$\curlyvee$};
\node[emp] at (0, -1) {$6\ \vert\ 2\ \vert\ 1\ 3\ \vert\ 4\ \vert\ 8\ \vert\ 9\ \vert\ 7\ 5$};
\node[emp] at (-1, -2) {$(4, 1; 8, 1; 9, 2; 6, 1; 0, 1; 1, 1; 3, 2)$};
\node[emp] at (0, -3) {$\curlyvee$};
\node[emp] at (0, -4) {$6\ \vert\ 2\ \vert\ 1\ \vert\ 3\ \vert\ 4\ \vert\ 9\ 8\ \vert\ 7\ 5$};
\node[emp] at (-1, -5) {$(4, 1; 8, 1; 9, 1; 0, 1; 0, 1; 0, 2; 3, 2)$};
\node[emp] at (0, -6) {$\curlyvee$};
\node[emp] at (0, -7) {$6\ \vert\ 2\ \vert\ 3\ \vert\ 1\ \vert\ 4\ \vert\ 9\ 8\ \vert\ 7\ 5$};
\node[emp] at (-1, -8) {$(4, 1; 8, 1; 0, 1; 9, 1; 0, 1; 0, 2; 3, 2)$};

\node[emp] at (12, 5) {$2\ 6\ \vert\ 1\ 3\ \vert\ 4\ \vert\ 9\ 8\ \vert\ 5\ \vert\ 7$};
\node[emp] at (12.6, 4) {$(8, 2; 9, 2; 6, 1; 0, 2; 5, 1; 0, 1)$};
\node[emp] at (12, 3) {$\curlyvee$};
\node[emp] at (12, 2) {$6\ 2\ \vert\ 1\ 3\ \vert\ 4\ \vert\ 9\ 8\ \vert\ 5\ \vert\ 7$};
\node[emp] at (12.6, 1) {$(4, 2; 9, 2; 6, 1; 0, 2; 5, 1; 0, 1)$};
\node[emp] at (12, 0) {$\curlyvee$};
\node[emp] at (12, -1) {$6\ \vert\ 2\ \vert\ 1\ 3\ \vert\ 4\ \vert\ 9\ 8\ \vert\ 5\ 7$};
\node[emp] at (12.6, -2) {$(4, 1; 8, 1; 9, 2; 6, 1; 0, 2; 5, 2)$};
\node[emp] at (12, -3) {$\curlyvee$};
\node[emp] at (12, -4) {$6\ \vert\ 2\ \vert\ 1\ 3\ \vert\ 4\ \vert\ 9\ 8\ \vert\ 7\ 5$};
\node[emp] at (12.6, -5) {$(4, 1; 8, 1; 9, 2; 6, 1; 0, 2; 3, 2)$};

\draw (4, 8)--(8, -4);
\end{tikzpicture}
\end{center}
\caption{In $\cell(9, 2)$, the cells $f = 6\ \vert\ 2\ \vert\ 1\ 3\ \vert\ 4\ \vert\ 9\ 8\ \vert\ 5\ \vert\ 7$ and $g = 6\ \vert\ 2\ \vert\ 1\ 3\ \vert\ 4\ \vert\ 9\ 8\ \vert\ 7\ 5$ are paired, because $f$ is the greatest among faces of $g$ (shown in left column) and $g$ is the least among cofaces of $f$ (shown in right column).  In the picture, below each symbol appears the corresponding value of the key function.}\label{fig-face-coface}
\end{figure}

\begin{proof}
We describe the pairing on the remaining cells, and then verify that it comes from the total ordering.  An example of two paired cells is shown in Figure~\ref{fig-face-coface}.  Suppose that $f$ is a cell such that there are two consecutive singletons in increasing order, but in the string of blocks preceding those, the two conditions for being critical are met.  Let $g$ be the cell in which those two consecutive singletons are combined such that the resulting $2$--element block has its elements in decreasing order.  Then the discrete vector field contains $[f, g]$.  From the reverse point of view, suppose that $g$ is a cell such that there is a $2$--element block with elements in decreasing order, not immediately preceded by a lesser singleton, such that in the string of blocks preceding this $2$--element block, the two conditions for being critical are met.  Let $f$ be the cell in which this $2$--element block is split into two singletons in increasing order.  Then the discrete vector field contains $[f, g]$.

We need to show that for such a pair $[f, g]$, $f$ is the greatest face of $g$ and $g$ is the least coface of $f$.  To show the former, suppose that $f$ comes from splitting the $k$th block of $g$.  Any face of $g$ that comes from splitting an earlier block is less than $f$, because if the block is ascending it gets shorter and its first entry cannot decrease (properties (2) and (3) of Lemma~\ref{lem-key}), and if the block is descending it is a follower and remains a follower while getting shorter (property (1) of Lemma~\ref{lem-key}).  Because the $k$th block of $g$ is descending but not a follower, any face of $g$ that comes from splitting a later block, or the face that comes from splitting the same block in the other way (i.e., such that the resulting singletons remain in decreasing order), is less than $f$ because the first element of the $k$th block is greater than that of $f$ (property (2) of Lemma~\ref{lem-key}).  Thus $f$ is the greatest face of $g$.

Similarly, suppose that $[f, g]$ is a pair in the discrete vector field and $g$ comes from combining the $k$th and $(k+1)$st blocks of $f$, which must then be ascending singletons.  Any coface of $f$ that comes from combining two earlier blocks is greater than $g$, because the blocks are decreasing singletons, so the first of the two blocks gets longer and its first element cannot increase (properties (2) and (3) of Lemma~\ref{lem-key}).  Any coface of $f$ that comes from combining two later blocks, or from combining the $k$th and $(k+1)$st blocks in the other way (i.e., in ascending order), is greater than $g$ because the $k$th block of that coface has a lesser first element than $g$ and is not a follower (property (2) of Lemma~\ref{lem-key}).  Thus $g$ is the least coface of $f$.

\end{proof}

A similar analysis gives the set of critical cells for $\desc(n, w)$.

\begin{lemma}\label{lem-desc-crit}
If a cell in $\desc(n, w)$ is critical with respect to the discrete gradient vector field that comes from the total ordering from Lemma~\ref{lem-key}, then the cell has the following properties:
\begin{enumerate}
\item Every two consecutive singletons are in decreasing order.
\item Every non-singleton block has $w$ elements and is a follower.
\end{enumerate}
\end{lemma}

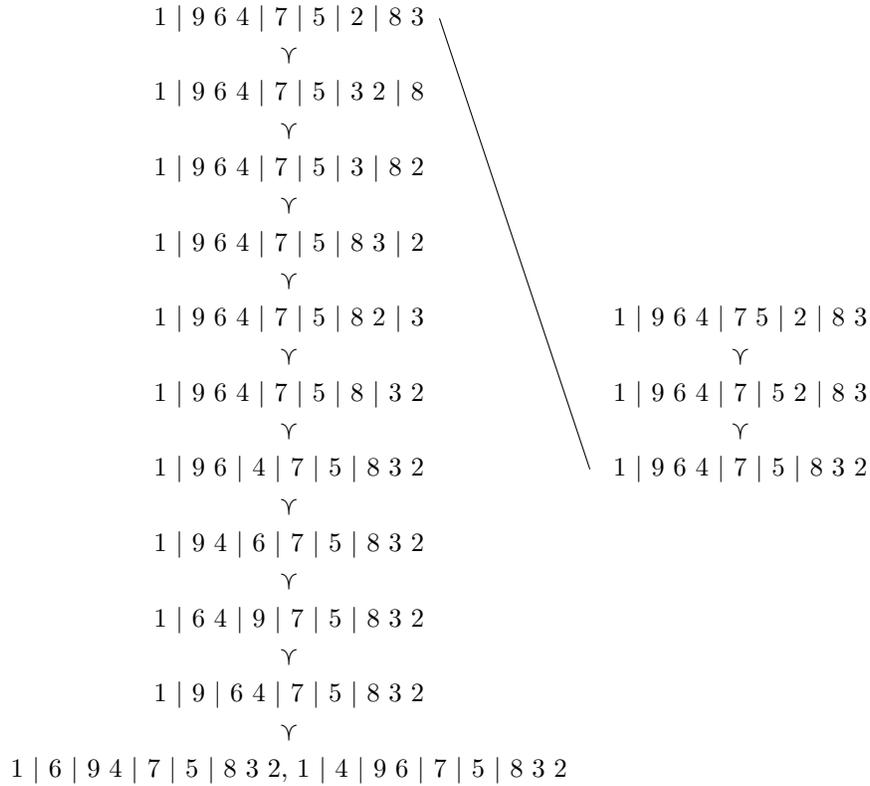
\begin{figure}
\begin{center}
\begin{tikzpicture}[scale = .5, emp/.style={inner sep = 0pt, outer sep = 0pt}]
\node[emp] at (0, 21) {$1\ \vert\ 9\ 6\ 4\ \vert\ 7\ \vert\ 5\ \vert\ 2\ \vert\ 8\ 3$};
\node[emp] at (0, 20) {$\curlyvee$};
\node[emp] at (0, 19) {$1\ \vert\ 9\ 6\ 4\ \vert\ 7\ \vert\ 5\ \vert\ 3\ 2\ \vert\ 8$};
\node[emp] at (0, 18) {$\curlyvee$};
\node[emp] at (0, 17) {$1\ \vert\ 9\ 6\ 4\ \vert\ 7\ \vert\ 5\ \vert\ 3\ \vert\ 8\ 2$};
\node[emp] at (0, 16) {$\curlyvee$};
\node[emp] at (0, 15) {$1\ \vert\ 9\ 6\ 4\ \vert\ 7\ \vert\ 5\ \vert\ 8\ 3\ \vert\ 2$};
\node[emp] at (0, 14) {$\curlyvee$};
\node[emp] at (0, 13) {$1\ \vert\ 9\ 6\ 4\ \vert\ 7\ \vert\ 5\ \vert\ 8\ 2\ \vert\ 3$};
\node[emp] at (0, 12) {$\curlyvee$};
\node[emp] at (0, 11) {$1\ \vert\ 9\ 6\ 4\ \vert\ 7\ \vert\ 5\ \vert\ 8\ \vert\ 3\ 2$};
\node[emp] at (0, 10) {$\curlyvee$};
\node[emp] at (0, 9) {$1\ \vert\ 9\ 6\ \vert\ 4\ \vert\ 7\ \vert\ 5\ \vert\ 8\ 3\ 2$};
\node[emp] at (0, 8) {$\curlyvee$};
\node[emp] at (0, 7) {$1\ \vert\ 9\ 4\ \vert\ 6\ \vert\ 7\ \vert\ 5\ \vert\ 8\ 3\ 2$};
\node[emp] at (0, 6) {$\curlyvee$};
\node[emp] at (0, 5) {$1\ \vert\ 6\ 4\ \vert\ 9\ \vert\ 7\ \vert\ 5\ \vert\ 8\ 3\ 2$};
\node[emp] at (0, 4) {$\curlyvee$};
\node[emp] at (0, 3) {$1\ \vert\ 9\ \vert\ 6\ 4\ \vert\ 7\ \vert\ 5\ \vert\ 8\ 3\ 2$};
\node[emp] at (0, 2) {$\curlyvee$};
\node[emp] at (0, 1) {$1\ \vert\ 6\ \vert\ 9\ 4\ \vert\ 7\ \vert\ 5\ \vert\ 8\ 3\ 2$, $1\ \vert\ 4\ \vert\ 9\ 6\ \vert\ 7\ \vert\ 5\ \vert\ 8\ 3\ 2$};

\node[emp] at (12, 13) {$1\ \vert\ 9\ 6\ 4\ \vert\ 7\ 5\ \vert\ 2\ \vert\ 8\ 3$};
\node[emp] at (12, 12) {$\curlyvee$};
\node[emp] at (12, 11) {$1\ \vert\ 9\ 6\ 4\ \vert\ 7\ \vert\ 5\ 2\ \vert\ 8\ 3$};
\node[emp] at (12, 10) {$\curlyvee$};
\node[emp] at (12, 9) {$1\ \vert\ 9\ 6\ 4\ \vert\ 7\ \vert\ 5\ \vert\ 8\ 3\ 2$};

\draw (4, 21)--(8, 9);
\end{tikzpicture}
\end{center}
\caption{To find the discrete gradient vector field on $\desc(n, w)$, we use the ordering of cells from $\cell(n)$, but the resulting pairing is different.  Pictured are (left column) the faces of $1\ \vert\ 9\ 6\ 4\ \vert\ 7\ \vert\ 5\ \vert\ 8\ 3\ 2$ and (right column) the cofaces of $1\ \vert\ 9\ 6\ 4\ \vert\ 7\ \vert\ 5\ \vert\ 2\ \vert\ 8\ 3$ in $\desc(9, 3)$.}\label{fig-desc-morse}
\end{figure}

\begin{proof}
The pairing on the remaining cells is defined as follows.  Suppose that $f$ is a cell such that there is a singleton immediately followed by a block of size less than $w$, for which every element is greater than the singleton, and in the string of blocks preceding those, the two conditions for being critical are met.  Let $g$ be the cell in which the singleton is combined with the following block.  Then the discrete vector field contains $[f, g]$.  From the reverse point of view, suppose that $g$ is a cell such that there is a non-singleton block that is not preceded by a lesser singleton, and in the string of preceding blocks, the two conditions for being critical are met.  Let $f$ be the cell in which this non-singleton block is split into two blocks, the first of which is the least element as a singleton block.  Then the discrete vector field contains $[f, g]$.

Suppose that $[f, g]$ is in the discrete vector field and $f$ comes from splitting the $k$th block of $g$.  To show that $f$ is the greatest face of $g$, consider the result of splitting any earlier block of $g$.  Because $g$ looks like a critical cell at that stage, that block has size $w$ and is a follower; after splitting, it is shorter and is still a follower, so the key is less than that of $g$ (property (1) of Lemma~\ref{lem-key}); in contrast, among ways to split the $k$th block another way, or to split a later block, $f$ is the greatest because it is the only one for which the $k$th block begins with that least element of the $k$th block of $g$ (property (2) of Lemma~\ref{lem-key}).

To show that $g$ is the least coface of $f$, consider the result of combining any earlier blocks of $f$.  Because $f$ looks like a critical cell at that stage, the two blocks would be non-follower singletons in decreasing order, so the combined block would be larger, not be a follower, and have the same first element as the first of the two singletons, giving a greater key than that of $f$ (property (3) of Lemma~\ref{lem-key}); in contrast, among ways to combine later blocks of $f$, $g$ is the least because it is the only one that increases the first element of the $k$th block (property (2) of Lemma~\ref{lem-key}).
\end{proof}

\section{Basis for homology}\label{sec-basis}

In the previous section we constructed discrete gradient vector fields on $\cell(n, 2)$ and $\desc(n, w)$ and described their critical cells.  In this section we construct $\mathbb{Z}$--bases for $H_*(\cell(n, 2))$ and $H_*(\desc(n, w))$ with one basis cycle per critical cell.

The following general lemma shows that it suffices to construct, for each critical cell $e$, a cycle $z(e)$ such that $e$ is its maximum cell and has coefficient $\pm 1$.  Then the rest of the section is devoted to the construction.

\begin{lemma}\label{lem-max-basis}
Let $X$ be any finite polyhedral cell complex with a total ordering on the cells, giving a discrete gradient vector field.  Suppose that for each critical cell $e$, there is a cycle $z(e)$ such that $e$ has coefficient $\pm 1$ in $z(e)$ and is the greatest cell appearing with nonzero coefficient in $z(e)$.  Then every homology class in $H_*(X)$ can be written uniquely as a $\mathbb{Z}$--linear combination of the homology classes of the cycles $z(e)$.
\end{lemma}

\begin{proof}
For any pair $[f, g]$ in the discrete vector field, we refer to $f$ as a ``match-up cell'' and refer to $g$ as a ``match-down cell''.  We also define $z'(f)$ to be the boundary of $g$; we know that $f$ is the greatest cell appearing in $z'(f)$, and that it has coefficient $\pm 1$ because the original complex $X$ is polyhedral.

First, we show that every $j$--cycle $z$ is a $\mathbb{Z}$--linear combination of cycles $z(e)$ and $z'(f)$, where $e$ ranges over the critical $j$--cells and $f$ ranges over the match-up $j$--cells.  This follows from the following observation: if a match-down cell $g$ is the greatest cell in a $j$--chain, then in the boundary of that chain, the corresponding match-up cell $f$ appears with nonzero coefficient, because $g$ is the least coface of $f$, so no other cell in the chain has $f$ as a face.  Thus, for any $j$--cycle $z$, the greatest cell of $z$ cannot be a match-down cell.  It is either a critical cell $e$ or a match-up cell $f$, so we subtract the appropriate multiple of $z(e)$ or $z'(f)$ to get a new cycle with lesser maximum.  Repeating this process gives us $z$ as a linear combination of cycles $z(e)$ and $z'(f)$, so because each $z'(f)$ is a boundary, this implies that $z$ is homologous to a linear combination of the cycles $z(e)$ only.

To show the uniqueness, we need to show that no nontrivial linear combination of cycles $z(e)$ is null-homologous.  Because the cycles $z(e)$ and $z'(f)$ have distinct maxima, they are linearly independent.  Thus, it suffices to show that if a $j$--cycle $z$ is a boundary, it is a linear combination of the boundaries $z'(f)$.  To see this, we look at the set of all $j+1$--chains.  The chains $z(e)$, $z'(f)$, and $g$ (as $e$ ranges over all critical $(j+1)$--cells, $f$ ranges over all match-up $(j+1)$--cells, and $g$ ranges over all match-up $(j+1)$--cells) form a $\mathbb{Z}$--basis for the set of all $(j+1)$--chains, because they have distinct maxima equal to the set of all $j$--cells.  When we apply the boundary map to this basis, the cycles $z(e)$ and $z'(f)$ map to zero, and the match-down cells $g$ map to the $j$--dimensional boundaries $z'(f)$.  Thus indeed every $j$--dimensional boundary is a linear combination of these boundaries $z'(f)$.

Thus, every homology class in $H_*(X)$ can be written as a $\mathbb{Z}$--linear combination of the homology classes of the cycles $z(e)$, and the combination is unique.
\end{proof}

In both $\cell(n, 2)$ and $\desc(n, w)$, each critical cell $e$ consists of a sequence of larger blocks, with descending sequences of singletons in between, and each larger block has a rigidly specified form.  The construction of the cycle $z(e)$ is based on this structure.  To make this precise, we define a bilinear concatenation product of chains in the following way.


%
%

In Section~\ref{sec-complex} we have defined injected cell and concatenation product of injected cells: given two cells with disjoint sets of labels, we can write the two cells with a vertical bar between them.  Just as we can apply an injection on $[n]$ to a cell in $\cell(n)$ (that is, we relabel the disks), we can also apply an injection to a $\mathbb{Z}$--linear combination of cells.  Applying injections commutes with taking boundary maps, and applying an injection to a cycle gives an \textit{\textbf{injected cycle}}.  If two injected cycles $z_1 = \sum_{i} \alpha_{1, i} f_{1, i}$ and $z_2 = \sum_{j} \alpha_{2, j} f_{2, j}$ have disjoint labels, then we define their concatenation product to be
\[z_1 \vert z_2 = \sum_{i, j} \alpha_{1, i}\alpha_{2, j} \cdot f_{1, i} \vert f_{2, j},\]
which the Leibniz rule from Section~\ref{sec-complex} implies is also an injected cycle.

We typically restrict our attention to order-preserving injections on $[n]$; the total ordering on $\cell(n)$ does not respect arbitrary injections, but it does respect order-preserving injections.  Every critical cell $e$ in $\cell(n, 2)$ can be written uniquely as a concatenation product of some number of images of the cells $1$, $1\ 2$, and $1\ \vert\ 3\ 2$ under order-preserving injections, and every critical cell in $\desc(n, w)$ can be written uniquely as a concatenation product of images of cells $1$ and $1\ \vert\ (w+1)\ w \ \cdots\ 3\ 2$ under order-preserving injections; we refer to these images as \textit{\textbf{irreducible}} critical injected cells.  We can associate a cycle $z(e)$ to each critical cell $e$ by first doing so on the irreducibles.  We set $z(1) = 1$ and $z(1\ 2) = 1\ 2 + 2\ 1$.  We set $z(1\ \vert\ 3\ 2)$ to be the boundary in $\cell(3)$ of the cell $3\ 2\ 1$, which is $z(1\ \vert\ 3\ 2) = 1\ \vert\ 3\ 2 + 3\ 1\ \vert\ 2 + 3\ \vert\ 2\ 1 - 3\ 2\ \vert\ 1 - 2\ \vert\ 3\ 1 - 2\ 1\ \vert\ 3$.  We set $z(1\ \vert\ (w+1)\ w\ \cdots\ 3\ 2)$ to be the boundary in $\cell(w+1)$ of the cell $(w+1)\ w\ \cdots\ 3\ 2\ 1$.  

\begin{figure}
\begin{center}
\begin{tikzpicture}[scale=.8, emp/.style={inner sep = 0pt, outer sep = 0pt}, >=stealth]
\draw (-.25, 0)--(3.8, 0);
\draw (-.25, 2)--(3.8, 2);

\node (d3) at (1, .55) {$3$};
\node (d5) at (1+.5, .55+.866) {$5$};
\node (d2) at (1-.5, .55+.866) {$2$};
\node (d4) at (3, .5) {$4$};
\node (d1) at (3, 1.5) {$1$};

\draw (d3) circle (.5);
\draw (d5) circle (.5);
\draw (d2) circle (.5);
\draw (d4) circle (.5);
\draw (d1) circle (.5);

\draw[->] (1.53033008588991, 0.596669914110089) arc (-45:-15:.75);
\draw[->] (1.19411428382689, 1.85144436971680) arc (75:105:.75);
\draw[->] (0.275555630283199, 0.932885716173110) arc (195:225:.75);

\draw[->] (3+.52, .7) arc (-30:30:.6);
\draw[->] (3-.52, 1.3) arc (150:210:.6);
\end{tikzpicture}\end{center}
\caption{The cycle $z(2\ \vert\ 5\ 3\ \vert\ 1\ 4)$ is defined as the concatenation product $z(2\ \vert\ 5\ 3)\ \vert\ z(1\ 4)$, the result of putting the cycles $z(2\ \vert\ 5\ 3)$ and $z(1\ 4)$ side by side in the strip.}\label{fig-concatenation}
\end{figure}
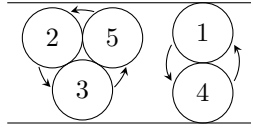

By requiring $z$ to commute with order-preserving injections on the labels and with taking concatenation products---i.e., $z(f_1 \vert f_2) = z(f_1) \vert z(f_2)$, as in Figure~\ref{fig-concatenation}---we obtain a definition of $z$ for all critical cells in $\cell(n, 2)$ and in $\desc(n, w)$.

Next we check the hypothesis of Lemma~\ref{lem-max-basis}.

\begin{lemma}\label{lem-crit-max}
Let $e$ be any critical cell of $\cell(n, 2)$ or $\desc(n, w)$.  Then in the cycle $z(e)$, the cell $e$ has coefficient $\pm 1$, and it is the greatest cell that appears in $z(e)$ with nonzero coefficient.
\end{lemma}

\begin{proof}
First we show the statement where $e$ is an irreducible critical injected cell.  In our ordering, $1\ 2$ is greater than $2\ 1$, so $1\ 2$ is the greatest cell of $z(1\ 2)$.  In any $\cell(n)$, every coefficient in the boundary of any cell is $\pm 1$ (or $0$).  And, in our ordering, the greatest face of $(w+1)\ w\ \cdots\ 3\ 2\ 1$ is $1\ \vert\ (w+1)\ w\ \cdots\ 3\ 2$, which is the only face that begins with $1$ (property (3) of Lemma~\ref{lem-key}).  Using $w = 2$ this applies to $z(1\ \vert\ 3\ 2)$ also, so the lemma statement is true whenever $e$ is an irreducible critical injected cell.

Next we consider concatenation products and apply induction on the number of irreducibles.  Suppose that $e = e_1 \ \vert\ e_2$, where $e_1$ and $e_2$ are critical (that is, they are the images of critical cells under order-preserving injections), and suppose that the lemma statement is true for both $e_1$ and $e_2$.  By definition, the coefficient of $e$ in $z(e_1 \ \vert\ e_2) = z(e_1)\ \vert\ z(e_2)$ is the product of the coefficient of $e_1$ in $z(e_1)$ and the coefficient of $e_2$ in $z(e_2)$, so the coefficient is $\pm 1$.  Let $f_1$ be any cell appearing in $z(e_1)$, and let $f_2$ be any cell appearing in $z(e_2)$.  We need to show that $e = e_1 \ \vert\ e_2$ is at least as great as $f_1\ \vert\ f_2$, knowing that $e_1 \succeq f_1$ and $e_2 \succeq f_2$.  Indeed, if $e_1 \succ f_1$, then by the properties in Lemma~\ref{lem-key} we know that $e_1 \ \vert\ e_2 \succ f_1\ \vert\ f_2$ no matter what $f_2$ is.  And, if $f_1 = e_1$, then the fact that $e_2 \succ f_2$ implies that $e_1 \ \vert\ e_2 \succ e_1 \vert\ f_2$.  (Note that the characterization of critical cells implies that the first block of $e_2$ cannot be a follower in $e$.  The first block of $f_2$ may become a follower in $e_1\ \vert\ f_2$, in which case we use property (4) from Lemma~\ref{lem-key}.)

By induction on the number of irreducibles, $e$ is the greatest cell in $z(e)$ whether or not it is irreducible.
\end{proof}

Putting together Lemmas \ref{lem-cell-crit}, \ref{lem-desc-crit}, \ref{lem-max-basis}, and~\ref{lem-crit-max}, we have proved the following theorem.

\begin{theorem}\label{thm-basis}
A basis for $H_*(\cell(n, 2))$ is given by the classes of the cycles $z(e)$, where $e$ ranges over all cells with the following properties:
\begin{enumerate}
\item Every two consecutive singletons are in decreasing order.
\item If a given $2$--element block has its elements in decreasing order, then the block is a follower.
\end{enumerate}
A basis for $H_*(\desc(n, w))$ is given by the classes of the cycles $z(e)$, where $e$ ranges over all cells with the following properties:
\begin{enumerate}
\item Every two consecutive singletons are in decreasing order.
\item Every non-singleton block has $w$ elements and is a follower.
\end{enumerate}
\end{theorem}

\section{Generating an $\FI_d$--module}\label{sec-abstract}

In this section we prove a general lemma about $\FI_d$--modules.  If we want to prove that a given sequence of abelian groups is an $\FI_d$--module, many verifications are needed: we need to specify a group homomorphism for each morphism in $\FI_d$, and we need to prove that compositions that are equal in $\FI_d$ give equal group homomorphisms.  To streamline such a proof, we can write every morphism as a composition of permutations and what we call high-insertion maps, which correspond to the various inclusions $[n] \hookrightarrow [n+1]$.  Lemma~\ref{lem-general-fid} below states which compatibility properties we need to check, in order for the permutations and high-insertion maps to specify an $\FI_d$--module.

First we review the precise definition of $\FI_d$--module, from~\cite{Ramos17}.  The category $\FI_d$ has one object $[n] = \{1, \ldots, n\}$ for each natural number $n$.  The morphisms are pairs $(\varphi, c)$, where $\varphi$ is an injection, say, from $[n]$ to $[m]$, and $c$ is a $d$--coloring on the complement of the image of $\varphi$; that is, $c$ is a map from $[m] \setminus \varphi([n])$ to a set of size $d$, which in this paper we choose to be $\{0, 1, \ldots, d-1\}$.  The morphisms compose as illustrated in Figure~\ref{fig-compose}: for each element colored by the first morphism, in the composition, the image of that element under the second morphism is the one that gets that color.  (In the picture, the color of a given element is shown in a diamond just above the element.)  More formally, if $(\varphi, c) \co [n_1] \rightarrow [n_2]$ and $(\varphi', c') \co [n_2] \rightarrow [n_3]$ are two morphisms, then we have
\[(\varphi', c') \circ (\varphi, c) = (\varphi' \circ \varphi, c''),\]
where $c''(i)$ is equal to $c'(i)$ if $i \not\in \varphi'([n_2])$, and is equal to $c(\varphi'^{-1}(i))$ if $i \in \varphi'([n_2])$.

\begin{figure}
\begin{center}
\begin{tikzpicture} [scale=.5, emp/.style={inner sep = 0pt, outer sep = 0pt}, >=stealth]
\node[emp] at (1, 7) {$1$};
\node[emp] at (2, 7) {$2$};
\node[emp] at (3, 7) {$3$};
\draw (.5, 6.5)--(3.5, 6.5)--(3.5, 7.5)--(.5, 7.5)--cycle;
\draw (1.5, 6.5)--(1.5, 7.5) (2.5, 6.5)--(2.5, 7.5);

\draw[->] (1, 6.5)--(3, 4.5);
\draw[->] (2, 6.5)--(2, 4.5);
\draw[->] (3, 6.5)--(4, 4.5);
\draw (1, 4.5)--(1.5, 5)--(1, 5.5)--(.5, 5)--cycle;
\node[emp] at (1, 5) {$2$};

\node[emp] at (1, 4) {$1$};
\node[emp] at (2, 4) {$2$};
\node[emp] at (3, 4) {$3$};
\node[emp] at (4, 4) {$4$};
\draw (.5, 3.5)--(4.5, 3.5)--(4.5, 4.5)--(.5, 4.5)--cycle;
\draw (1.5, 3.5)--(1.5, 4.5) (2.5, 3.5)--(2.5, 4.5) (3.5, 3.5)--(3.5, 4.5);

\draw[->] (1, 3.5)--(2, 1.5);
\draw[->] (2, 3.5)--(1, 1.5);
\draw[->] (3, 3.5)--(4, 1.5);
\draw[->] (4, 3.5)--(5, 1.5);
\draw (3, 1.5)--(3.5, 2)--(3, 2.5)--(2.5, 2)--cycle;
\node[emp] at (3, 2) {$1$};

\node[emp] at (1, 1) {$1$};
\node[emp] at (2, 1) {$2$};
\node[emp] at (3, 1) {$3$};
\node[emp] at (4, 1) {$4$};
\node[emp] at (5, 1) {$5$};
\draw (.5, .5)--(5.5, .5)--(5.5, 1.5)--(.5, 1.5)--cycle;
\draw (1.5, .5)--(1.5, 1.5) (2.5, .5)--(2.5, 1.5) (3.5, .5)--(3.5, 1.5) (4.5, .5)--(4.5, 1.5);


\node[emp] at (6, 4) {$=$};

\node[emp] at (7+1, 5.5) {$1$};
\node[emp] at (7+2, 5.5) {$2$};
\node[emp] at (7+3, 5.5) {$3$};
\draw (7+.5, 5)--(7+3.5, 5)--(7+3.5, 6)--(7+.5, 6)--cycle;
\draw (7+1.5, 5)--(7+1.5, 6) (7+2.5, 5)--(7+2.5, 6);

\draw[->] (7+1, 5)--(7+4, 3);
\draw[->] (7+2, 5)--(7+1, 3);
\draw[->] (7+3, 5)--(7+5, 3);
\draw (7+2, 3)--(7+2.5, 3.5)--(7+2, 4)--(7+1.5, 3.5)--cycle;
\node[emp] at (7+2, 3.5) {$2$};
\draw[fill=white] (7+3, 3)--(7+3.5, 3.5)--(7+3, 4)--(7+2.5, 3.5)--cycle;
\node[emp] at (7+3, 3.5) {$1$};

\node[emp] at (7+1, 2.5) {$1$};
\node[emp] at (7+2, 2.5) {$2$};
\node[emp] at (7+3, 2.5) {$3$};
\node[emp] at (7+4, 2.5) {$4$};
\node[emp] at (7+5, 2.5) {$5$};
\draw (7+.5, 2)--(7+5.5, 2)--(7+5.5, 3)--(7+.5, 3)--cycle;
\draw (7+1.5, 2)--(7+1.5, 3) (7+2.5, 2)--(7+2.5, 3) (7+3.5, 2)--(7+3.5, 3) (7+4.5, 2)--(7+4.5, 3);
\end{tikzpicture}
\end{center}
\caption{To compose two morphisms in $\FI_d$, we have $(\varphi', c') \circ (\varphi, c) = (\varphi' \circ \varphi, c'')$, where $c''(i)$ is equal to $c'(i)$ if $i$ is not in the image of $\varphi'$ (for instance, $i = 3$ has color $1$ in the example shown) and is equal to $c(\varphi'^{-1}(i))$ if $i$ is in the image of $\varphi'$ (for instance, $i = 2$ has color $2$ in the composition because $c(1) = 2$ and $\varphi'(1) = 2$).}\label{fig-compose}
\end{figure}
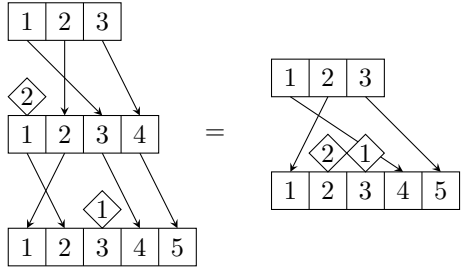

An \textit{\textbf{$\FI_d$--module}} $M$ over a commutative ring $k$ is defined to be a functor from $\FI_d$ to $k$--modules; that is, we have a $k$--module $M_n$ for each $n$, and for each $(\varphi, c) \co [n] \rightarrow [m]$, we have a corresponding $k$--module map $(\varphi, c)_* \co M_n \rightarrow M_m$.  In the present paper we use $k = \mathbb{Z}$.  An $\FI_d$--module is \textit{\textbf{finitely generated}} if there exists a finite set of elements $x_1, \ldots, x_r \in \bigsqcup_{n = 1}^\infty M_n$ such that the only $\FI_d$--submodule of $M$ containing $x_1, \ldots, x_r$ is $M$ itself.

Any $\FI_d$--module is determined by the permutation action on each $M_n$, along with the $d$ different maps from $M_n$ to $M_{n+1}$ that correspond to taking the inclusion from $[n]$ into $[n+1]$ and coloring the element $n+1$ each of the $d$ different colors.  We refer to these latter maps as the \textit{\textbf{high-insertion maps}}.  Notationally, we denote by $[i_k] \co M_n \rightarrow M_{n+1}$ the $k$th high-insertion map, which colors element $n+1$ with the color $k$.  We denote by $[\sigma] \co M_n \rightarrow M_n$ the permutation map corresponding to a permutation $\sigma \in S_n$.

In order for a choice of permutation action and high-insertion maps to correspond to an $\FI_d$--module, we need to check some compatibility properties.  We say that ``high-insertion maps commute with permutations'' if for every color $k$, every $n$, and every $\sigma \in S_n$, we have
\[[i_k] \circ [\sigma] = [\widetilde{\sigma}] \circ [i_k],\]
where $\widetilde{\sigma} \in S_{n+1}$ fixes the element $n+1$ and permutes the other elements according to $\sigma$.  We say that ``insertions are unordered'' if for every pair of colors $k, \ell$ and every $n$, we have the following relation of maps from $M_n$ to $M_{n+2}$:
\[[(n+1\ \ n+2)] \circ [i_k] \circ [i_\ell] = [i_\ell] \circ [i_k].\]
Here $(n+1 \ \ n+2)$ denotes the permutation in $S_{n+2}$ that transposes the greatest two elements.  The following lemma says that checking these two properties is enough to define an $\FI_d$--module.

\begin{lemma}\label{lem-general-fid}
Suppose we have modules $M_n$, with $S_n$--actions on the various $M_n$ and $d$ high-insertion maps from each $M_n$ to $M_{n+1}$.  If ``high-insertion maps commute with permutations'' and ``insertions are unordered'', then the compositions of these maps form an $\FI_d$--module.
\end{lemma}

\begin{proof}
For each morphism $(\varphi, c)$ in $\FI_d$, we need to define a map $(\varphi, c)_*\co M_n \rightarrow M_m$.  We already have a definition when $\varphi$ is a permutation, that is, when $(\varphi, c) = (\sigma, \cdot)$, where $\sigma \in S_n$ and $\cdot$ denotes an empty coloring.  In this case $(\sigma, \cdot)_* = [\sigma]$.  The high-insertion maps describe what happens when $\varphi$ is the inclusion map from $[n]$ to $[n+1]$, that is, when $(\varphi, c) = (i, \ n+1 \mapsto k)$, where $i \co [n] \hookrightarrow [n+1]$ is the inclusion.  In this case we set $(i, \ n+1\mapsto k)_* = [i_k]$.

\begin{figure}
\begin{center}
\begin{tikzpicture} [scale=.5, emp/.style={inner sep = 0pt, outer sep = 0pt}, >=stealth]
\node[emp] at (1, 5) {$1$};
\node[emp] at (2, 5) {$2$};
\node[emp] at (3, 5) {$3$};
\draw (.5, 4.5)--(3.5, 4.5)--(3.5, 5.5)--(.5, 5.5)--cycle;
\draw (1.5, 4.5)--(1.5, 5.5) (2.5, 4.5)--(2.5, 5.5);

\draw[->] (1, 4.5)--(1, 3.5);
\draw[->] (2, 4.5)--(2, 3.5);
\draw[->] (3, 4.5)--(3, 3.5);
\draw (4, 3.5)--(4.5, 4)--(4, 4.5)--(3.5, 4)--cycle;
\node[emp] at (4, 4) {$2$};
\draw (5, 3.5)--(5.5, 4)--(5, 4.5)--(4.5, 4)--cycle;
\node[emp] at (5, 4) {$1$};

\node[emp] at (1, 3) {$1$};
\node[emp] at (2, 3) {$2$};
\node[emp] at (3, 3) {$3$};
\node[emp] at (4, 3) {$4$};
\node[emp] at (5, 3) {$5$};
\draw (.5, 2.5)--(5.5, 2.5)--(5.5, 3.5)--(.5, 3.5)--cycle;
\draw (1.5, 2.5)--(1.5, 3.5) (2.5, 2.5)--(2.5, 3.5) (3.5, 2.5)--(3.5, 3.5) (4.5, 2.5)--(4.5, 3.5);

\draw[->] (1, 2.5)--(4, 1.5);
\draw[->] (2, 2.5)--(1, 1.5);
\draw[->] (3, 2.5)--(5, 1.5);
\draw[->] (4, 2.5)--(2, 1.5);
\draw[->] (5, 2.5)--(3, 1.5);

\node[emp] at (1, 1) {$1$};
\node[emp] at (2, 1) {$2$};
\node[emp] at (3, 1) {$3$};
\node[emp] at (4, 1) {$4$};
\node[emp] at (5, 1) {$5$};
\draw (.5, .5)--(5.5, .5)--(5.5, 1.5)--(.5, 1.5)--cycle;
\draw (1.5, .5)--(1.5, 1.5) (2.5, .5)--(2.5, 1.5) (3.5, .5)--(3.5, 1.5) (4.5, .5)--(4.5, 1.5);


\node[emp] at (-1, 3) {$=$};

\node[emp] at (-8+1, 4.5) {$1$};
\node[emp] at (-8+2, 4.5) {$2$};
\node[emp] at (-8+3, 4.5) {$3$};
\draw (-8+.5, 4)--(-8+3.5, 4)--(-8+3.5, 5)--(-8+.5, 5)--cycle;
\draw (-8+1.5, 4)--(-8+1.5, 5) (-8+2.5, 4)--(-8+2.5, 5);

\draw[->] (-8+1, 4)--(-8+4, 2);
\draw[->] (-8+2, 4)--(-8+1, 2);
\draw[->] (-8+3, 4)--(-8+5, 2);
\draw (-8+2, 2)--(-8+2.5, 2.5)--(-8+2, 3)--(-8+1.5, 2.5)--cycle;
\node[emp] at (-8+2, 2.5) {$2$};
\draw[fill=white] (-8+3, 2)--(-8+3.5, 2.5)--(-8+3, 3)--(-8+2.5, 2.5)--cycle;
\node[emp] at (-8+3, 2.5) {$1$};

\node[emp] at (-8+1, 1.5) {$1$};
\node[emp] at (-8+2, 1.5) {$2$};
\node[emp] at (-8+3, 1.5) {$3$};
\node[emp] at (-8+4, 1.5) {$4$};
\node[emp] at (-8+5, 1.5) {$5$};
\draw (-8+.5, 1)--(-8+5.5, 1)--(-8+5.5, 2)--(-8+.5, 2)--cycle;
\draw (-8+1.5, 1)--(-8+1.5, 2) (-8+2.5, 1)--(-8+2.5, 2) (-8+3.5, 1)--(-8+3.5, 2) (-8+4.5, 1)--(-8+4.5, 2);

\end{tikzpicture}
\end{center}
\caption{Any morphism in $\FI_d$ can be decomposed as a sequence of high-insertion maps, followed by a permutation that preserves the order of the newly inserted elements.  The morphism in the figure is equal to $(1\ 4\ 2)(3\ 5) \circ i_1 \circ i_2$.}\label{fig-decompose}
\end{figure}
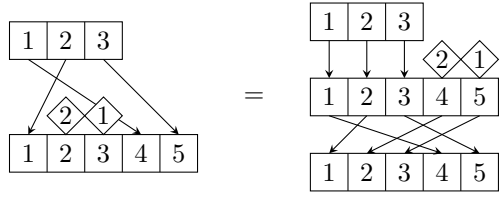

Given an arbitrary morphism $(\varphi, c)$ in $\FI_d$, as in Figure~\ref{fig-decompose} we can write the injection $\varphi \co [n] \rightarrow [m]$ uniquely as $\sigma_\varphi \circ i^{m-n}$, such that $\sigma_{\varphi} \in S_m$ is order-preserving on the set $[m] \setminus [n]$, and $i^{m-n}$ denotes the composition of inclusions $i \co [n] \hookrightarrow [n+1]$, $i \co [n+1] \hookrightarrow [n+2]$, and so on.  In other words, $\sigma_{\varphi}$ takes the same values on $[n]$ as $\varphi$, and maps $[m] \setminus [n]$ to the complement of the image of $\varphi$ in order.  Looking at the coloring $c$ on $[m] \setminus \varphi([n])$, we let $c_1, \ldots, c_{m-n}$ denote the values of $c$ in order; to be precise, we have $c_i = c(\sigma_{\varphi}(i))$ for $i = n+1, \ldots, m$.  Then we have
\[(\varphi, c) = (\sigma_\varphi, \cdot) \circ (i, \ m \mapsto c_{m-n}) \circ \cdots \circ (i, \ n+1 \mapsto c_1),\]
so we should define
\begin{align*}
(\varphi, c)_* &= (\sigma_\varphi, \cdot)_* \circ (i, \ m \mapsto c_{m-n})_* \circ \cdots \circ (i, \ n+1 \mapsto c_1)_* \\
&= [\sigma_\varphi] \circ [i_{c_{m-n}}] \circ \cdots \circ [i_{c_1}].
\end{align*}

To check functoriality, we need to check that if we have another sequence of permutations and high-insertions that composes to $(\varphi, c)$ in $\FI_{d}$, then the corresponding maps on the various modules $M_n$ compose to $(\varphi, c)_*$.  Given an arbitrary sequence of permutations and high-insertion maps, the property that ``high-insertion maps commute with permutations'' implies that we can push all the permutations to the left past the high-insertion maps, without changing the composition map, to get a composition of permutations followed by a composition of high-insertion maps.  Using the fact that the permutations in each $S_n$ form a group action on $M_n$, we can replace the composition of permutations by a single permutation.  

Thus, to prove that we have an $\FI_{d}$--module, it suffices to show that if
\[(\varphi, c) = (\sigma', \cdot) \circ (i, \ m \mapsto c'_{m-n}) \circ \cdots \circ (i, \ n+1 \mapsto c'_{1}),\]
then we have
\[[\sigma'] \circ [i_{c'_{m-n}}] \circ \cdots \circ [i_{c'_1}] = [\sigma_\varphi] \circ [i_{c_{m-n}}] \circ \cdots \circ [i_{c_1}].\]
Because $\sigma'$ and $\sigma_{\varphi}$ both take the same values on $[n]$ as $\varphi$, we can write $\sigma' = \sigma_\varphi \circ \sigma''$, where $\sigma''$ only permutes $[m]\setminus [n]$.  Thus, canceling $[\sigma_\varphi]$ from both sides it suffices to show that we have
\[[\sigma''] \circ [i_{c'_{m-n}}] \circ \cdots \circ [i_{c'_{1}}] = [i_{c_{m-n}}] \circ \cdots \circ [i_{c_{1}}].\]

This identity comes from the property that ``insertions are unordered''.  Specifically, we can use induction on $m-n$.  If $m-n = 1$, there is nothing to prove.  Otherwise, let $n+k = (\sigma'')^{-1}(n+1)$; that is, in the alternative composition, element $n+k$ gets inserted with color $c'_k = c_1$ and then $\sigma''$ changes its number to $n+1$.  By the ``insertions are unordered'' property we can write
\[[i_{c'_k}]\circ [i_{c'_{k-1}}] = [(n+k\ n+k-1)] \circ [i_{c'_{k-1}}] \circ [i_{c'_k}],\]
then
\[[i_{c'_k}] \circ [i_{c'_{k-2}}] = [(n+k-1\ n+k-2)] \circ [i_{c'_{k-2}}] \circ [i_{c'_k}],\]
and so on, until the composition ends with $[i_{c'_k}]$ on the right.  Then, applying the ``high-insertion maps commute with permutations'' property, we can move all the transpositions to the left to make the composition
\[\sigma'' \circ (n+k \ n+k-1)\circ (n+k-1\ n+k-2) \circ \cdots \circ (n+2\ n+1),\]
which is equal to
\[\sigma'' \circ (n+k\ n+k-1\ \cdots \ n+2\ n+1),\]
a permutation that fixes $n+1$.  Denoting this new permutation by $\sigma'''$, we have
\[[\sigma''] \circ [i_{c'_{m-n}}] \circ \cdots \circ [i_{c'_{1}}] = [\sigma'''] \circ [i_{c'_{m-n}}] \circ \cdots \circ [i_{c'_{k+1}}] \circ [i_{c'_{k-1}}] \circ \cdots \circ [i_{c'_1}] \circ [i_{c'_k}],\]
and we know that $c'_k = c_1$ by how we have selected $k$.  Applying the inductive hypothesis, we have
\[[\sigma'''] \circ [i_{c'_{m-n}}] \circ \cdots \circ [i_{c'_{k+1}}] \circ [i_{c'_{k-1}}] \circ \cdots \circ [i_{c'_1}] = [i_{c_{m-n}}] \circ \cdots \circ [i_{c_{2}}],\]
and so composing with $[i_{c_1}]$ on the right, we obtain the desired equality.
\end{proof}

\section{$\FI_d$--module for disks in a strip}\label{sec-finish}

The goal of this section is to prove the following two theorems, which are the main theorems of this paper.

\begin{theorem}\label{thm-fid}
For any $j$, the homology groups $H_j(\cell(n, 2)) = H_j(\config(n, 2))$ form a finitely generated $\FI_{j+1}$--module over $\mathbb{Z}$.
\end{theorem}

\begin{theorem}\label{thm-nok-fid}
For any $j \geq 0$ and $w \geq 1$, the homology groups $H_j(\desc(n, w))$ are zero unless $j$ is a multiple of $w-1$.  If $j = b(w-1)$ for some integer $b$, then the homology groups $H_j(\desc(n, w))$ form a finitely generated $\FI_{b+1}$--module over $\mathbb{Z}$, and thus for $k=w+1$, the no--$k$--equal homology groups $H_j(\no_k(n, \mathbb{R}))$ also form a finitely generated $\FI_{b+1}$--module.
\end{theorem}

Theorem~\ref{thm-basis} implies that $H_j(\desc(n, w)) = 0$ if $j$ is not a multiple of $w-1$, because in this case $\desc(n, w)$ has no critical $j$--cells.  And, we know from Section~\ref{sec-complex} that the cell complexes $\cell(n, 2)$ and $\desc(n, w)$ are homotopy equivalent to the configuration space $\config(n, 2)$ and the no-$k$-equal space $\no_{w+1}(n, \mathbb{R})$, respectively.  Thus, in both cases it remains to specify the permutation action and the high-insertion maps, to check the compatibility properties from the hypothesis of Lemma~\ref{lem-general-fid}, and to verify that the resulting $\FI_d$--module is finitely generated.

The permutation actions on $H_*(\cell(n, 2))$ and $H_*(\desc(n, w))$ come from the permutation actions on $\cell(n, 2)$ and $\desc(n, w)$, which correspond to the permutation actions on $\config(n, 2)$ and $\no_{w+1}(n, \mathbb{R})$ by permuting the labels.  Specifically, for each cell in $\cell(n, 2)$, we apply the permutation to the numbers in that symbol, giving another cell in $\cell(n, 2)$.  For each cell in $\desc(n, w)$, we apply the permutation to the numbers in that symbol, and then rearrange the numbers within each block so that they are in descending order.  For each permutation $\sigma \in S_n$, we denote the corresponding maps on homology by $[\sigma]$.  We note that the permutations do not respect the basis for homology given in Theorem~\ref{thm-basis}; applying a permutation to a basis cycle $z(e)$ may give a cycle that is homologous to a linear combination of several basis cycles.

We define the high-insertion maps in terms of \textit{\textbf{barriers}}, which roughly are the non-singleton blocks.  Specifically, as before we write each critical cell $e$ as the (unique) concatenation product of images of the cells $1$, $1\ 2$, $1\ \vert\ 3\ 2$, and $1\ \vert\ (w+1)\ w\ \cdots \ 3\ 2$ under order-preserving injections.  We consider each image of $1\ 2$, $1\ \vert\ 3\ 2$, and $1\ \vert\ (w+1)\ w\ \cdots\ 3\ 2$ to be a barrier.  For any critical $j$--cell in $\cell(n, 2)$, the number of barriers is $j$, and for any critical $j$--cell in $\desc(n, w)$, the number of barriers is $b = \frac{j}{w-1}$.

\begin{figure}
\begin{center}
\begin{tikzpicture}[scale=.5, emp/.style={inner sep = 0pt, outer sep = 0pt}, >=stealth]
\draw (5.5-.25, 5+0)--(5.5+5, 5+0);
\draw (5.5-.25, 5+2)--(5.5+5, 5+2);

\node (d3) at (.51+5.5+1, 5+.55) {$3$};
\node (d5) at (.51+5.5+1+.5, 5+.55+.866) {$5$};
\node (d2) at (.51+5.5+1-.5, 5+.55+.866) {$2$};
\node (d4) at (-.51+5.5+4.1, 5+.5) {$4$};
\node (d1) at (-.51+5.5+4.1, 5+1.5) {$1$};

\draw (d3) circle (.5);
\draw (d5) circle (.5);
\draw (d2) circle (.5);
\draw (d4) circle (.5);
\draw (d1) circle (.5);

\draw[->] (.51+5.5+1.53033008588991, 5+0.596669914110089) arc (-45:-15:.75);
\draw[->] (.51+5.5+1.19411428382689, 5+1.85144436971680) arc (75:105:.75);
\draw[->] (.51+5.5+0.275555630283199, 5+0.932885716173110) arc (195:225:.75);

\draw[->] (-.51+5.5+4.1+.52, 5+.7) arc (-30:30:.6);
\draw[->] (-.51+5.5+4.1-.52, 5+1.3) arc (150:210:.6);

\draw[->] (5.5+2.375, 4.75)--(5.5+2.375, 2.25);
\node[emp] at (5.5+2.375, 3.5) [label=left:{$i_1$}] {};
\draw[->] (5.5+1.375, 4.75)--(2.375, 2.25);
\node[emp] at (2.25+1.875, 3.5) [label=left:{$i_0$}] {};
\draw[->] (5.5+3.375, 4.75)--(11+2.375, 2.25);
\node[emp] at (7.75+3.875, 3.5) [label=right:{$i_2$}] {};


\draw (-.25, 0)--(5, 0);
\draw (-.25, 2)--(5, 2);

\node (d3) at (1.3+1, .55) {$3$};
\node (d5) at (1.3+1+.5, .55+.866) {$5$};
\node (d2) at (1.3+1-.5, .55+.866) {$2$};
\node (d4) at (4.1, .5) {$4$};
\node (d1) at (4.1, 1.5) {$1$};
\node (d6) at (.65, .5) {$6$};

\draw (d3) circle (.5);
\draw (d5) circle (.5);
\draw (d2) circle (.5);
\draw (d4) circle (.5);
\draw (d1) circle (.5);
\draw (d6) circle (.5);

\draw[->] (1.3+1.53033008588991, 0.596669914110089) arc (-45:-15:.75);
\draw[->] (1.3+1.19411428382689, 1.85144436971680) arc (75:105:.75);
\draw[->] (1.3+0.275555630283199, 0.932885716173110) arc (195:225:.75);

\draw[->] (4.1+.52, .7) arc (-30:30:.6);
\draw[->] (4.1-.52, 1.3) arc (150:210:.6);


\draw (5.5-.25, 0)--(5.5+5, 0);
\draw (5.5-.25, 2)--(5.5+5, 2);

\node (d3) at (5.5+1, .55) {$3$};
\node (d5) at (5.5+1+.5, .55+.866) {$5$};
\node (d2) at (5.5+1-.5, .55+.866) {$2$};
\node (d4) at (5.5+4.1, .5) {$4$};
\node (d1) at (5.5+4.1, 1.5) {$1$};
\node (d6) at (5.5+2.7, .5) {$6$};

\draw (d3) circle (.5);
\draw (d5) circle (.5);
\draw (d2) circle (.5);
\draw (d4) circle (.5);
\draw (d1) circle (.5);
\draw (d6) circle (.5);

\draw[->] (5.5+1.53033008588991, 0.596669914110089) arc (-45:-15:.75);
\draw[->] (5.5+1.19411428382689, 1.85144436971680) arc (75:105:.75);
\draw[->] (5.5+0.275555630283199, 0.932885716173110) arc (195:225:.75);

\draw[->] (5.5+4.1+.52, .7) arc (-30:30:.6);
\draw[->] (5.5+4.1-.52, 1.3) arc (150:210:.6);


\draw (11-.25, 0)--(11+5, 0);
\draw (11-.25, 2)--(11+5, 2);

\node (d3) at (11+1, .55) {$3$};
\node (d5) at (11+1+.5, .55+.866) {$5$};
\node (d2) at (11+1-.5, .55+.866) {$2$};
\node (d4) at (11+2.8, .5) {$4$};
\node (d1) at (11+2.8, 1.5) {$1$};
\node (d6) at (11+4.2, .5) {$6$};

\draw (d3) circle (.5);
\draw (d5) circle (.5);
\draw (d2) circle (.5);
\draw (d4) circle (.5);
\draw (d1) circle (.5);
\draw (d6) circle (.5);

\draw[->] (11+1.53033008588991, 0.596669914110089) arc (-45:-15:.75);
\draw[->] (11+1.19411428382689, 1.85144436971680) arc (75:105:.75);
\draw[->] (11+0.275555630283199, 0.932885716173110) arc (195:225:.75);

\draw[->] (-1.3+11+4.1+.52, .7) arc (-30:30:.6);
\draw[->] (-1.3+11+4.1-.52, 1.3) arc (150:210:.6);
\end{tikzpicture}
\end{center}
\caption{The $k$th high-insertion map $i_k$ inserts a singleton just after the $k$th barrier, with label greater than all the existing labels.}\label{fig-high-options}
\end{figure}
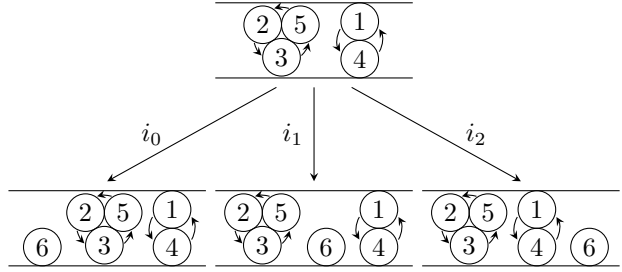

For $0 \leq k \leq j$ in the case of $\cell(n, 2)$, and for $0 \leq k \leq b$ in the case of $\desc(n, w)$, the $k$th high-insertion map is defined as follows and is depicted in Figure~\ref{fig-high-options}.  Given a critical cell $e$ of $\cell(n, 2)$, to find $i_k(e)$ we insert a block containing only the number $n+1$, right after the $k$th barrier of $e$ (or as the first block, if $k = 0$).  We observe that the result is also a critical cell.  Thus, these maps give rise to maps on homology, $[i_k] \co H_j(\cell(n, 2)) \rightarrow H_j(\cell(n+1, 2))$ and $[i_k] \co H_j(\desc(n, w)) \rightarrow H_j(\desc(n+1, w))$:  Given a homology class, we write it in terms of the basis cycles $z(e)$, and then replace each $z(e)$ by $z(i_k(e))$.

The proof of compatibility between the permutation action and the high-insertion maps is based on the following three useful properties of barriers:
\begin{enumerate}
\item ``Number of barriers is preserved by permutation'': If $f$ is a critical cell and $\sigma$ is a permutation, and we write $\sigma(z(f))$ in terms of the basis as
\[\sigma(z(f)) = \sum_i \alpha_i\cdot z(e_i),\]
then each critical cell $e_i$ has the same number of barriers as $f$ has.
\item ``Only barriers obstruct singletons'': Let $e$ be a critical cell without any barrier, and let $i$ be a single number.  Then the cycles $z(e)\  \vert\ i$ and $i\ \vert\ z(e)$ are homologous.
\item ``Critical cells can concatenate at a barrier'': If $e_1$ and $e_2$ are critical cells such that $e_1$ ends with a barrier, then $e_1\ \vert\ e_2$ is a critical cell.
\end{enumerate}

Property~(1) is true because in $\cell(n, 2)$ or in $\desc(n, w)$, every critical cell of a given dimension has the same number of barriers.  Property~(3) is true according to our characterization of critical cells in Theorem~\ref{thm-basis}.  To prove Property~(2), we observe that the only critical cells without barriers have only singleton blocks, so it suffices to show that we can permute consecutive singletons, as in the following lemma and in Figure~\ref{fig-swap-singleton}.

\begin{lemma}\label{lem-leibniz}
Let $z_1$ and $z_2$ be injected cycles with disjoint sets of labels, and let $p$ and $q$ be numbers not appearing in $z_1$ and $z_2$.  Then the concatenation products $z_1\ \vert\ p\ \vert\ q\ \vert\ z_2$ and $z_1\ \vert\ q\ \vert\ p\ \vert\ z_2$ are homologous.
\end{lemma}

\begin{figure}
\begin{center}
\begin{tikzpicture}[scale=.8, emp/.style={inner sep = 0pt, outer sep = 0pt}, >=stealth]
\node[text=gray!50] (d22) at (2.2, 1.5) {$2$};
\draw[gray!50] (d22) circle (.5);
\draw (-.5, 0)--(3.8, 0);
\draw (-.5, 2)--(3.8, 2);

\draw[->] (0+.52, .7) arc (-30:30:.6);
\draw[->] (0-.52, 1.3) arc (150:210:.6);
\node (d1) at (0, 1.5) {$1$};
\node (d4) at (0, .5) {$4$};
\node (d2) at (1.1, .5) {$2$};
\node (d3) at (2.2, .5) {$3$};
\draw (d1) circle (.5);
\draw (d4) circle (.5);
\draw (d2) circle (.5);
\draw (d3) circle (.5);
\draw[->, gray!50]  (1.24737205583712, 1.05000000000000) arc (150:120:1.1);
\draw[->, gray!50] (2.75000000000000, 1.45262794416288) arc (60:30:1.1);

\end{tikzpicture}\hspace{.2in}
\begin{tikzpicture}[scale=.8, emp/.style={inner sep = 0pt, outer sep = 0pt}, >=stealth]
\draw (-.5, 0)--(3.8, 0);
\draw (-.5, 2)--(3.8, 2);

\draw[->] (0+.52, .7) arc (-30:30:.6);
\draw[->] (0-.52, 1.3) arc (150:210:.6);
\node (d1) at (0, 1.5) {$1$};
\node (d4) at (0, .5) {$4$};
\node (d2) at (3.3, .5) {$2$};
\node (d3) at (2.2, .5) {$3$};
\draw (d1) circle (.5);
\draw (d4) circle (.5);
\draw (d2) circle (.5);
\draw (d3) circle (.5);

\end{tikzpicture}
\end{center}
\caption{The cycles $z(1\ 4\ \vert\ 2\ \vert\ 3)$ and $z(1\ 4\ \vert\ 3\ \vert\ 2)$ are homologous, because their difference is the boundary of the chain $z(1\ 4)\ \vert\ 2\ 3$.  More generally, we can permute consecutive singleton blocks in a cycle without changing the homology class.}\label{fig-swap-singleton}
\end{figure}
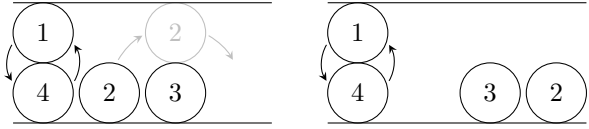

\begin{proof}
The proof follows from the more general fact that the homology class of any concatenation product of cycles is preserved by replacing a factor by a homologous factor.  This is because of the Leibniz rule from Section~\ref{sec-complex}: for any injected cells $f_1$ and $f_2$ with disjoint sets of labels, we have
\[\del(f_1 \ \vert\ f_2) = \del f_1 \ \vert\ f_2 + (-1)^{b(f_1)} f_1\ \vert\ \del f_2,\]
where $b(f_1)$ denotes the number of blocks in $f_1$.  The two concatenation products $z_1\ \vert\ p\ \vert\ q\ \vert\ z_2$ and $z_1\ \vert\ q\ \vert\ p\ \vert\ z_2$ differ by the boundary of the chain $z_1\ \vert\ q\ p\ \vert\ z_2$, so they are homologous.
\end{proof}

Using these three properties of barriers, we can verify the two compatibility properties between the permutations and the high-insertion maps.

\begin{lemma}\label{lem-permute-insert}
For each $j$, the $S_n$--actions and high-insertion maps on the homology groups $H_j(\cell(n, 2))$ and $H_j(\desc(n, w))$ have the property that ``high-insertion maps commute with permutations'' and the property that ``insertions are unordered''.
\end{lemma}

\begin{figure}
\begin{center}
\begin{tikzpicture}[scale=1, emp/.style={inner sep = 0pt, outer sep = 0pt}, >=stealth]
\node[emp] at (0, 0) {$-z(1\vert 32 \vert 4) + z(1\vert 23\vert 4) - z(23\vert 1 \vert 4) \sim -z(1\vert 32 \vert 4) + z(1 \vert 23\vert 4) - z(23\vert 4 \vert 1)$};
\node[emp] at (-3, 1.5) {$z(1\vert 32) \vert 4$};
\node[emp] at (3, 1.5) {$-z(1\vert 32) + z(1\vert 23) - z(23 \vert 1)$};
\node[emp] at (0, 3) {$z(1\vert 32)$};
\draw[->] (-.5, 2.75)--(-2.5, 1.75);
\node[emp] at (-1.5, 2.25) [label=above:{$i_1$}] {};
\draw[->] (.5, 2.75)--(2.5, 1.75);
\node[emp] at (1.5, 2.25) [label=above:{$(2\ 3)$}] {};
\draw[->] (-2.5, 1.25)--(-2.3, .25);
\node[emp] at (-2.4, .75) [label=right:{$(2\ 3)$}] {};
\draw[->] (2.5, 1.25)--(2.3, .25);
\node[emp] at (2.4, .75) [label=left:{$i_1$}] {};
\end{tikzpicture}
\end{center}
\caption{The property that ``high-insertion maps commute with permutations'' is not quite true for cycles, but it is true for homology classes.}\label{fig-commute}
\end{figure}

\begin{proof}
First we verify the first property, that if $\sigma \in S_n$, then for all $k$ we have
\[[i_k] \circ [\sigma] = [\widetilde{\sigma}] \circ [i_k],\]
where $\widetilde{\sigma}$ is the corresponding permutation in $S_{n+1}$.  Figure~\ref{fig-commute} lays out what we need to prove in a specific example.  It suffices to check the desired relation on the basis cycles $z(e)$, where $e$ is a critical cell.  We can write $e$ as $e_1 \ \vert\ e_2$, where $e_1$ ends with the $k$th barrier of $e$, so that $i_k(e) = e_1 \ \vert \ n+1 \ \vert\  e_2$.  Let $z_1$ and $z_2$ be the injected cycles resulting from applying $\sigma$ to $z(e_1)$ and $z(e_2)$, so that we have $[\sigma] z(e) = z_1 \ \vert\ z_2$.  By definition we have
\[([\widetilde{\sigma}] \circ [i_k]) z(e) = [\widetilde{\sigma}]z(e_1 \ \vert\ n+1\ \vert\ e_2) = z_1\  \vert\ n+1 \ \vert\ z_2,\]
and we want to show that this cycle is homologous to $([i_k] \circ [\sigma]) z(e) = [i_k](z_1 \ \vert\ z_2)$.

Applying a high-insertion map to $z_1 \ \vert\ z_2$ requires writing that cycle in terms of the basis cycles.  Suppose that $z_1$ is homologous to $\sum_{i} \alpha_i z(e_i)$.  Then $z_1 \ \vert\ z_2$ is homologous to $\sum_i \alpha_i z(e_i) \ \vert\ z_2$, and $z_1 \ \vert\ n+ 1\ \vert\ z_2$ is homologous to $\sum_{i} \alpha_i z(e_i) \ \vert\ n+ 1\ \vert\ z_2$, so it suffices to show, for any $i_0$, that we have (using the $\sim$ symbol for homologous cycles)
\[[i_k](z(e_{i_0}) \ \vert\ z_2) \sim z(e_{i_0})\  \vert\ n+ 1\  \vert\ z_2.\]
Using the ``number of barriers is preserved by permutation'' property, we know that $e_{i_0}$ has exactly $k$ barriers.  We write $e_{i_0}$ as $c \ \vert\ t$, where $c$ (``core'') ends with the $k$th barrier of $e_{i_0}$, and $t$ (``tail'') consists of the remaining blocks, containing no barrier.  We can write $z(t) \ \vert\ z_2$ in terms of the basis as $\sum_i \beta_i z(e_i)$.  Then, using the ``critical cells can concatenate at a barrier'' property, each $c \ \vert\ e_i$ is a critical cell, so we have
\[[i_k](z(e_{i_0}) \ \vert\ z_2) = [i_k](z(c) \ \vert\ z(t) \ \vert\ z_2) = \sum_i \beta_i \cdot [i_k](z(c) \ \vert\ z(e_i)) = \sum_i \beta_i \cdot z(c)\ \vert\ n+1\ \vert\ z(e_1).\]
To show this is homologous to $z(e_{i_0})\ \vert\ n+ 1\ \vert\ z_2$, we use the ``only barriers obstruct singletons'' property.  This implies that $z(t)\ \vert\ n+1$ is homologous to $n+1\ \vert\ z(t)$, so we have
\[z(e_{i_0})\  \vert\ n+ 1\ \vert\ z_2 = z(c)\ \vert\ z(t)\ \vert\ n+ 1\ \vert\ z_2 \sim z(c)\ \vert\ n+ 1\ \vert\ z(t) \ \vert \ z_2 \sim \sum_{i} \beta_i \cdot z(c)\ \vert\ n+ 1\ \vert\ z(e_1).\]
Thus for each $i_0$ we have
\[[i_k](z(e_{i_0}) \ \vert\ z_2) \sim z(e_{i_0})\ \vert\ n+ 1\ \vert\ z_2,\]
and so we have
\[[i_k](z_1 \ \vert\ z_2) \sim z_1\ \vert\ n+ 1\ \vert\ z_2,\]
as desired.  This completes the proof that ``high-insertion maps commute with permutations''.

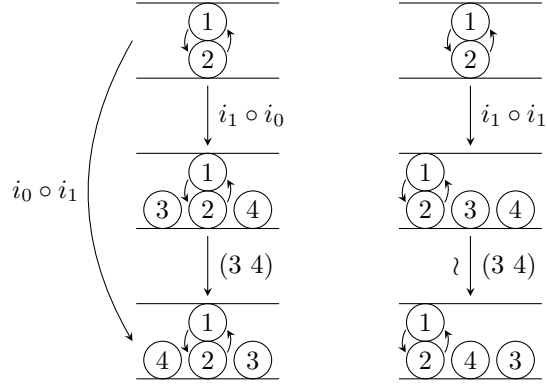
\begin{figure}
\begin{center}
\begin{tikzpicture}[scale=.5, emp/.style={inner sep = 0pt, outer sep = 0pt}, >=stealth]
\draw (-1.9, 0)--(1.9, 0);
\draw (-1.9, 2)--(1.9, 2);
\node[emp] (d1) at (0, 1.5) {$1$};
\node[emp] (d2) at (0, .5) {$2$};
\draw[->] (.52, .7) arc (-30:30:.6);
\draw[->] (-.52, 1.3) arc (150:210:.6);
\node[emp] (d3) at (-1.2, .5) {$3$};
\node[emp] (d4) at (1.2, .5) {$4$};

\draw (d1) circle (.5);
\draw (d2) circle (.5);
\draw (d3) circle (.5);
\draw (d4) circle (.5);


\draw (-1.9, 4+0)--(1.9, 4+0);
\draw (-1.9, 4+2)--(1.9, 4+2);
\node[emp] (d1) at (0, 4+1.5) {$1$};
\node[emp] (d2) at (0, 4+.5) {$2$};
\draw[->] (.52, 4+.7) arc (-30:30:.6);
\draw[->] (-.52, 4+1.3) arc (150:210:.6);

\draw (d1) circle (.5);
\draw (d2) circle (.5);


\draw (-1.9, -4+0)--(1.9, -4+0);
\draw (-1.9, -4+2)--(1.9, -4+2);
\node[emp] (d1) at (0, -4+1.5) {$1$};
\node[emp] (d2) at (0, -4+.5) {$2$};
\draw[->] (.52, -4+.7) arc (-30:30:.6);
\draw[->] (-.52, -4+1.3) arc (150:210:.6);
\node[emp] (d3) at (-1.2, -4+.5) {$4$};
\node[emp] (d4) at (1.2, -4+.5) {$3$};

\draw (d1) circle (.5);
\draw (d2) circle (.5);
\draw (d3) circle (.5);
\draw (d4) circle (.5);


\draw[->] (0, 3.8)--(0, 2.2);
\node[emp] at (0, 3) [label=right:{$i_1 \circ i_0$}] {};
\draw[->] (0, -.2)--(0, -1.8);
\node[emp] at (0, -1) [label=right:{$(3\ 4)$}] {};
\path[->] (-2, 5) edge[bend right] node [left] {$i_0 \circ i_1$} (-2, -3);

\end{tikzpicture}\hspace{.5in}
\begin{tikzpicture}[scale=.5, emp/.style={inner sep = 0pt, outer sep = 0pt}, >=stealth]
\draw (-1.9, 0)--(1.9, 0);
\draw (-1.9, 2)--(1.9, 2);
\node[emp] (d1) at (-1.2, 1.5) {$1$};
\node[emp] (d2) at (-1.2, .5) {$2$};
\draw[->] (-1.2+.52, .7) arc (-30:30:.6);
\draw[->] (-1.2-.52, 1.3) arc (150:210:.6);
\node[emp] (d4) at (0, .5) {$3$};
\node[emp] (d3) at (1.2, .5) {$4$};

\draw (d1) circle (.5);
\draw (d2) circle (.5);
\draw (d3) circle (.5);
\draw (d4) circle (.5);


\draw (-1.9, 4+0)--(1.9, 4+0);
\draw (-1.9, 4+2)--(1.9, 4+2);
\node[emp] (d1) at (0, 4+1.5) {$1$};
\node[emp] (d2) at (0, 4+.5) {$2$};
\draw[->] (.52, 4+.7) arc (-30:30:.6);
\draw[->] (-.52, 4+1.3) arc (150:210:.6);

\draw (d1) circle (.5);
\draw (d2) circle (.5);


\draw (-1.9, -4+0)--(1.9, -4+0);
\draw (-1.9, -4+2)--(1.9, -4+2);
\node[emp] (d1) at (-1.2, -4+1.5) {$1$};
\node[emp] (d2) at (-1.2, -4+.5) {$2$};
\draw[->] (-1.2+.52, -4+.7) arc (-30:30:.6);
\draw[->] (-1.2-.52, -4+1.3) arc (150:210:.6);
\node[emp] (d4) at (0, -4+.5) {$4$};
\node[emp] (d3) at (1.2, -4+.5) {$3$};

\draw (d1) circle (.5);
\draw (d2) circle (.5);
\draw (d3) circle (.5);
\draw (d4) circle (.5);


\draw[->] (0, 3.8)--(0, 2.2);
\node[emp] at (0, 3) [label=right:{$i_1 \circ i_1$}] {};
\draw[->] (0, -.2)--(0, -1.8);
\node[emp] at (0, -1) [label=right:{$(3\ 4)$}] [label=left:{$\wr$}] {};

\end{tikzpicture}
\end{center}
\caption{The ``insertions are unordered'' property says that when disks are inserted, the homology class of the result does not depend on the ordering of the insertions, as long as each disk goes between the right pair of barriers.}\label{fig-unordered}
\end{figure}

Next we check the property that ``insertions are unordered'', which is depicted in Figure~\ref{fig-unordered} and says that for all colors $k, \ell$ we have
\[[(n+1\ n+2)] \circ [i_k] \circ [i_\ell] = [i_\ell] \circ [i_k].\]
If $k \neq \ell$, then the equality holds on the level of cells; both sides result in a cell with the number $n+1$ inserted right after the $k$th barrier and $n+2$ inserted right after the $\ell$th barrier.  If $k=\ell$, then $i_k \circ i_k$ puts $n+2 \ \vert\ n+1$ right after the $k$th barrier, and $(n+1\ n+2) \circ i_k \circ i_k$ puts $n+1\ \vert\ n+2$ right after the $k$th barrier.  By the ``only barriers obstruct singletons'' property, these two resulting cells give homologous cycles.
\end{proof}

We now have the necessary pieces to prove our two main theorems.

\begin{proof}[Proof of Theorem~\ref{thm-fid}]
Lemma~\ref{lem-permute-insert} verifies the two hypotheses of Lemma~\ref{lem-general-fid}, which are that ``high-insertion maps commute with permutations'' and ``insertions are unordered''.  Then Lemma~\ref{lem-general-fid} implies that the compositions of these maps give an $\FI_d$--module for $d = j+1$.

We claim that a finite generating set for this $\FI_{j+1}$--module consists of the basis cycles $z(e)$ where $e$ is a critical cell for $n \leq 3j$.  Given any basis cycle $z(e)$ with $n > 3j$, we can write $e$ as the concatenation product of irreducible critical injected cells.  There are $j$ order-preserving images of the irreducibles $1\ 2$ and $1\ \vert\ 3\ 2$, and some nonzero number of additional singleton blocks.  Let $e'$ be the result of deleting these additional singleton blocks and shifting the numbers down so that they remain consecutive.  Then $z(e)$ is the result of applying some high-insertion maps and a permutation to $z(e')$---the permutation preserves the order of the numbers in $e'$---and $z(e')$ is in the proposed generating set.  Thus the finitely many basis cycles with $n \leq 3j$ do generate the $\FI_{j+1}$--module.
\end{proof}

\begin{proof}[Proof of Theorem~\ref{thm-nok-fid}]
If $j$ is not a multiple of $w-1$, there are no critical cells in dimension $j$ and thus no homology.  If $j = b(w-1)$, then Lemma~\ref{lem-permute-insert} and Lemma~\ref{lem-general-fid} imply that $H_j(\desc(n, w))$ is an $\FI_{b+1}$--module.  The critical cells with $n$ equal to $b(w+1)$---that is, those consisting only of barriers---form a finite generating set for the $\FI_{b+1}$--module.
\end{proof}

\section{Conclusion}\label{sec-conclusion}

To generalize the results of this paper to $\config(n, w)$ for $w > 2$ would be to prove the following conjecture.

\begin{conjecture}
For any $j$ and $w$, the homology groups $H_j(\config(n, w))$ form a finitely generated $\FI_d$--module for $d = 1 + \left\lfloor \frac{j}{w-1}\right\rfloor$.
\end{conjecture}

For various reasons this conjecture seems trickier to prove than the results of this paper.  The first difficulty is in finding a $\mathbb{Z}$--basis for $H_j(\config(n, w))$.  The strategy that produces the bases given in this paper goes roughly as follows.  There is a total ordering on the cells of $\cell(n)$ such that the critical cells of the discrete gradient vector field are in bijection with a $\mathbb{Z}$--basis of $H_*(\cell(n))$.  (To order, we modify the ``key'' function from Lemma~\ref{lem-key} so that there is no special case for follower blocks.  The critical cells are those where the first element of the block is the least element, and where furthermore the blocks appear in descending order of first element.)  When a cell in $\cell(n, 2)$ matches up to a higher-dimensional cell in $\cell(n) \setminus \cell(n, 2)$, the boundary of the higher-dimensional cell is a good choice of cycle to add to our basis.  In this way, for $w = 2$ we construct basis cycles as concatenation products of two kinds of cycles: those that generate the homology of $\cell(n)$, such as $z(1\ 2) = 1\ 2 + 2\ 1$, and those that are boundaries in $\cell(n, 2)$ of cells of $\cell(n) \setminus \cell(n, 2)$, such as $z(1\ \vert\ 3\ 2) = \del(3\ 2\ 1)$.

However, this strategy makes less sense for larger $w$.  For instance, in $\cell(6)$, the cell $4\ 5\ 6\ \vert\ 1\ 2\ 3$ is critical, whereas the cell $1\ 2\ 3\ \vert\ 4\ 5\ 6$ matches up to the cell $4\ 5\ 6\ 1\ 2\ 3$.  Should cell $1\ 2\ 3\ \vert\ 4\ 5\ 6$ be critical in $\cell(6, 4)$?  If so, its corresponding cycle cannot be $\del(4\ 5\ 6\ 1\ 2\ 3)$, which is in $\cell(6, 5)$ but is not in $\cell(6, 4)$.  For this reason, it is not clear whether the discrete Morse theory strategy can work for larger $w$.

\begin{question}
For $w > 2$, is $H_*(\config(n, w))$ a free abelian group?  If so, is there a discrete gradient vector field on $\cell(n, w)$ that has the same number of critical cells as the rank of $H_*(\config(n, w))$?
\end{question}

The second difficulty is in counting barriers in the various cycles.  In the theorems of this paper, for each space $\config(n, 2)$ or $\no_k(n, \mathbb{R})$, every cycle of a given dimension has the same number of barriers.  However, this is not true for $\config(n, w)$ with $w > 2$.  For instance, in $\config(8, 3)$ we can construct a $4$--cycle with no barriers by using four disjoint circling pairs, and we can also construct a $4$--cycle with two barriers by using two clusters of three disks and two fixed singleton disks.  The $\FI_d$--module structure depends completely on being able to recognize barriers in a consistent way.  How can we make the notion of barrier precise?  Roughly, we can say that a cycle $z$ on $n$ disks has at least one barrier if $(n+1)\ \vert\ z$ and $z\ \vert\ (n+1)$ are not homologous.  But, if a cycle is not a concatenation product, how can we count the barriers?

\begin{question}
Is there a collection of single-barrier cycles in $H_*(\config(n, w))$, such that there is an $S_n$--equivariant way of breaking arbitrary cycles (or homology classes) into sums of concatenation products of these single-barrier cycles?
\end{question}

Counting barriers is related to estimating the growth of the ranks of the homology groups.  The proofs in this paper imply that not only are $H_j(\config(n, 2))$ and $H_j(\no_k(n, \mathbb{R}))$ finitely generated $\FI_d$--modules for $d = j+1$ and $d = \frac{j}{k-2}$ respectively, but in fact the rank of each of these free abelian groups is equal to $d^n$ times a polynomial function of $n$.  In the case of $\config(n, w)$ for $w>2$ where cycles of the same dimension can have different numbers of barriers, we might expect a formula for the rank to include terms such as $(d-1)^n$ times a polynomial of $n$, $(d-2)^n$ times a polynomial, and so on, but in the setting of this paper we do not have these additional terms.

\begin{question}
Are the $\FI_d$-modules $H_j(\config(n, 2))$ and $H_j(\no_k(n, \mathbb{R}))$ free $\FI_d$--modules?
\end{question}

The $\FI_d$--modules are not freely generated by the basis cycles we give, because the permutation action does not take basis cycles to basis cycles (as in Figure~\ref{fig-commute}, for instance).  However, maybe there is another choice of generating set that would show the $\FI_d$--modules to be free.

Perhaps some new point of view can resolve the question of $\FI_d$--module structure on $H_j(\config(n, w))$ and produce related examples that may be of representation-theoretic interest.

\bibliographystyle{amsalpha}
\bibliography{disk-strip-bib}{}
\end{document}